\documentclass[reqno]{amsart}
\usepackage{amssymb, amsmath, amsthm, amscd, mathrsfs}

\usepackage{paralist}
\usepackage{xcolor}
\usepackage{color}
\usepackage{cite}
\usepackage{bigints}
\usepackage{dsfont}
\usepackage{empheq}
\usepackage{amssymb}
\usepackage{cases}
\usepackage{enumitem}
\allowdisplaybreaks

\usepackage{caption}
\usepackage{booktabs}

\usepackage{float}
\usepackage[ruled,vlined,linesnumbered,noresetcount]{algorithm2e}

\usepackage{hyperref}

\newtheorem{theorem}{Theorem}[section]
\newtheorem{lem}[theorem]{Lemma}
\newtheorem{cor}[theorem]{Corollary}
\newtheorem{prop}[theorem]{Proposition}
\newtheorem{defn}[theorem]{Definition}
\newtheorem{rmk}{Remark}

\numberwithin{equation}{section}

\def \d {\mathrm{d}}

\title[Maximal regularity for time-fractional Schrödinger equations]{Maximal regularity for time-fractional Schrödinger equations and application to nonlinear equations}

\author{S. E. Chorfi}
\author{F. Et-tahri}
\author{L. Maniar}
\author{M. Yamamoto}

\address{S. E. Chorfi, L. Maniar, Cadi Ayyad University, UCA, Faculty of Sciences Semlalia, Laboratory of Mathematics, Modeling and Automatic Systems, B.P. 2390, Marrakesh, Morocco}
\address{L. Maniar, The Vanguard Center, University Mohammed VI Polytechnic, Benguerir, Morocco}

\address{F. Et-tahri, Faculty of Sciences-Agadir, Lab-SIV, Ibn Zohr University, B.P. 8106, Agadir, Morocco}
\address{M. Yamamoto, Department of Mathematics, Graduate School of Mathematical Sciences, The University of Tokyo, Tokyo, 153-8914, Komaba, Meguro, Japan}
\address{
Department of Mathematics, Faculty of Science, Zonguldak Bulent Ecevit University, Zonguldak, 67100, Turkey
}

\email{s.chorfi@uca.ac.ma, maniar@uca.ac.ma}
\email{fouad.et-tahri@edu.uiz.ac.ma}
\email{myama@ms.u-tokyo.ac.jp}

\makeatletter %added for 2020 MSC
\@namedef{subjclassname@2020}{%
  \textup{2020} Mathematics Subject Classification}
\makeatother

 
\begin{document}
\begin{abstract}
We study the maximal regularity problem for abstract time-fractional Schrödinger equations $\partial_t^\alpha(u-u_0) -\mathrm{i} A u=f$, with a fractional derivative $\partial_t^\alpha$ of order $\alpha \in (0,1)$. We assume that $A$ is a self-adjoint operator with compact resolvent on a Hilbert space $H$. First, we prove the maximal $L^2$-regularity by leveraging properties of Mittag-Leffler functions with an imaginary argument. Compared to existing results for the subdiffusion equations, our proof avoids using the complete monotonicity of Mittag-Leffler functions, which seems difficult to prove within the setting of an imaginary argument. Then, we prove the maximal $L^p$-regularity for $p\in (1,\infty)$ using the operator-valued version of Mikhlin's multiplier theorem. Finally, we apply the maximal regularity results to prove the local well-posedness of quasilinear and semilinear time-fractional Schr\"odinger equations.
\end{abstract}

\maketitle

\section{Introduction and motivation}
The time-fractional Schrödinger equation has attracted significant interest because of its capacity to capture anomalous diffusion and memory effects in quantum systems. Applications include quantum mechanics, optics and photonics, condensed matter physics, and plasma physics; see, for instance, \cite{Gr19}, \cite{La18}, \cite{CC23}, and \cite{Wang07}, along with the references therein. The incorporation of fractional time dynamics breaks the usual unitarity and time-reversal symmetry, allowing the probability density to either dissipate or accumulate over time.

More precisely, we consider the time-fractional Schrödinger equation 
\begin{equation}\label{fseq0}
\partial_t^\alpha(u-u_0) -\mathrm{i} Au=f, \qquad t \in (0,T),
\end{equation}
where $T>0$ and $\partial_t^\alpha$ is a fractional derivative of order $\alpha\in (0,1)$ (see Section \ref{sec22}).

Inspired by time-fractional diffusion equations, M. Naber \cite{Na04} initiated the study of the time-fractional Schrödinger equation by giving a variant of \eqref{fseq0} with a factor $\mathrm{i}^{-\alpha}$ instead of $\mathrm{i}$. On the other hand, B.N. Achar et al. \cite{Ach13} obtained a special case of \eqref{fseq0} via the Feynman path integral method, allowing to recover the classical Schrödinger equation as $\alpha\to 1^-$. Moreover, H. Emamirad and A. Rougirel \cite{ER20} have proposed a time-fractional equation that preserves the classical quantum mechanical properties.

Maximal regularity for parabolic evolution equations means that $\partial_t u$ and $Au$ enjoy the same regularity as $f$, and this property has attracted considerable attention as a powerful tool for studying the well-posedness of nonlinear and non-autonomous equations via linearization and fixed-point arguments; see, for instance, \cite{PS16} and the references therein.

Maximal regularity for time-fractional evolution equations like \eqref{fseq0} means that $\partial_t^\alpha (u-u_0)$ and $Au$ enjoy the same regularity as $f$. It has recently garnered an increasing interest in studying regularity properties of models with anomalous or memory effects. It is well known that classical parabolic equations enjoy maximal regularity. It would therefore be natural to expect this property to be inherited by fractional parabolic equations. In this paper, we prove a surprising result on the maximal regularity of the time-fractional Schrödinger equation, which is not satisfied by the classical Schrödinger equation.

As pioneering works for parabolic equations, we can refer to P. E. Sobolevskii \cite{Sob64} and O. A. Ladyzhenskaya et al. \cite{La68}. Then, G. Da Prato and P. Grisvard \cite{DaP75} reformulated the problem using analytic semigroups, leading to two main approaches: the first one relies on singular integral operators; see, e.g., L. De Simon \cite{DeS64}, and M. Hieber and J. Prüss \cite{HP97} for Gaussian kernel estimates. The second approach is the operator-sum method, which yields interesting results via interpolation spaces \cite{DaP75}. We refer to the works of G. Dore and A. Venni, e.g., \cite{DV87} and \cite{Do93}. We also refer to H. Amann \cite{Am95}, P. C. Kunstmann and L. W. Weis \cite{KW04}, R. Denk et al. \cite{DHP03}, but the list is far from being exhaustive. As a recent reference, we mention W. Arendt and M. Sauter \cite{AS26} for generalized boundary conditions in time.

As first results for time-fractional equations, we can refer to the monograph by J. Pr\"uss \cite{Pr93}, where the author studied more general evolutionary integral equations. However, it is not clear whether one can obtain sharp regularity results for time-fractional equations. In \cite{CGL00} and \cite{CGL01}, P. Clement et al. proved maximal regularity results for abstract problems with sectorial operators. In \cite{Za05}, R. Zacher established maximal $L^p$-regularity results for abstract parabolic Volterra equations using operator-valued Fourier multiplier techniques. Moreover, in \cite{Za09}, the author developed a framework for weak solutions of abstract evolutionary integro‑differential equations in Hilbert spaces, establishing existence, uniqueness, and regularity results under general assumptions.

In \cite{Baz01}, E. Bazhlekova studied maximal $L^p$-regularity of abstract linear problems with the Riemann-Liouville derivative. The proofs are based on sums of operators and $\mathcal{R}$-boundedness. In the $L^2$-setting, L. Sakamoto and M. Yamamoto \cite{SY11} proved a maximal regularity result for fractional diffusion equations assuming an extra regularity on the inhomogeneous term. The proof relies on the complete monotonicity of Mittag-Leffler functions. In \cite{Gu19}, D. Guidetti proved maximal regularity results in spaces of continuous and Hölder continuous functions using interpolation spaces. For the subdiffusion equation with a time-dependent coefficient, we refer to B. Jin et al. \cite{JLZ19}. We also refer to A. Kubica et al. \cite{KRY20} for maximal $L^2$-regularity results for equations with variable coefficients. The reader may also refer to the monograph by B. Jin \cite{Ji21}, which discusses some maximal regularity results for subdiffusion. Recently, M. Achache \cite{Ach22} and \cite{Ach24} have proven maximal $L^p$-regularity results for non-autonomous problems in Hilbert spaces using the method of sesquilinear forms. Finally, we refer to G. Floridia et al. \cite{FGY25} for time-fractional evolution equations in Banach spaces. In particular, the authors have proven a maximal regularity result assuming an extra space regularity of the inhomogeneous term.

The article is organized as follows. In Section \ref{sec2}, we recall some preliminary results that will be used in the subsequent sections. Moreover, we present the necessary tools for fractional differentiation. Section \ref{sec4} is devoted to the notion of weak solution and some intermediate regularity results. Section \ref{sec5} investigates the maximal regularity for abstract fractional Schrödinger equations. First, we prove the maximal $L^2$-regularity without using the complete monotonicity of Mittag-Leffler functions. Then, we prove the maximal $L^p$-regularity for any $p\in (1,\infty)$ using a version of Mikhlin's multiplier theorem. In Section \ref{sec6}, we apply the maximal regularity to prove well-posedness results for quasilinear and semilinear Schr\"odinger equations. Finally, we give some final comments and perspectives.

\subsection*{Notations}
In what follows, we consider a Hilbert space $(H,\langle \cdot,\cdot \rangle)$, with the corresponding norm $\Vert\cdot\Vert:=\sqrt{\langle \cdot,\cdot \rangle}$. We assume throughout that $A:D(A)\subset H\rightarrow H$ is a self-adjoint operator, negative-definite, and with compact resolvent. We denote by $\sigma(-A) = \{\lambda_n\}_{n\in \mathbb{N}}\subset \mathbb{R}^{+}$ the spectrum of $-A$ and by $\{\varphi_n\}_{n\in \mathbb{N}}\subset D(A)$ an associated orthonormal basis of eigenvectors. Therefore, the operator $A$ satisfies
$$\langle Au,u \rangle <0 \qquad \forall u\in D(A)\setminus \{0\},$$
and the spectrum $\sigma(-A) = \{\lambda_n\}_{n\in \mathbb{N}}$ is such that
$$0<\lambda_1 \le \lambda_2 \le \cdots \le \lambda_n \le \cdots\to \infty.$$
We will denote by \(\mathcal{B}(H)\) the space of bounded linear operators on $H$ equipped with the operator norm. Then, $A^{-1}\in \mathcal{B}(H)$ and $\|A^{-1}\|\le \frac{1}{\lambda_1}$. Therefore, we can endow $D(A)$ with the equivalent norm
$$\|u\|_{D(A)}:=\|A u\|, \qquad  u\in D(A).$$
For any $z\in \rho(A)$ (the resolvent set), we denote by $R(z,A)=(z I_H-A)^{-1}$ the resolvent operator of $A$, where $I_H$ is the identity operator on $H$.

For any $p>1$, we denote by $q$ its conjugate exponent defined by $\dfrac{1}{p} + \dfrac{1}{q} = 1$. We will also use the weak Lebesgue spaces $L^{p,\infty}(0,T;V),$ $1 \le p < \infty,$ where $(V,\|\cdot\|_V)$ is a Hilbert space (see, e.g., \cite[Section 1.18.6]{Tr78}). The $L^{p,\infty}(0,T;V)$ quasi-norm of a measurable function $v : (0,T) \to V$ is given by
\[
\|v\|_{L^{p,\infty}(0,T;V)} := \sup_{\lambda>0} \, \lambda\Big|\Big\{ t\in (0,T)\colon \|v(t)\|_V>\lambda \Big\}\Big|^{\frac{1}{p}}.
\]
This functional fails to satisfy the triangle inequality and therefore does not
define a norm. However, for any $v \in L^p(0,T;V)$, one has
\[
\|v\|_{L^{p,\infty}(0,T;V)} \le \|v\|_{L^p(0,T;V)},
\]
which implies the continuous embedding $L^p(0,T;V) \subset L^{p,\infty}(0,T;V)$. Moreover, the spaces $L^{p,\infty}(0,T;V)$ are complete, and for $p>1$ they are Banach spaces.

For $\theta\in (0,1)$ and $1\leq p \leq \infty,$ we denote by $(X,Y)_{\theta,p}$ the real interpolation space for an interpolation couple $(X, Y)$ of Banach spaces; see, e.g., \cite[Section 1.5]{Tr78}. We write $X \hookrightarrow Y$ for a continuous embedding. Finally, we denote by $\overline{B}_X(r)$ the closed ball of radius $r>0$ centered at the origin in the space $X$.

\section{Preliminary results}\label{sec2}
Here, we present some preparatory results on Mittag-Leffler functions and fractional differentiation for future use.

\subsection{Mittag-Leffler functions}
We recall the two-parametric Mittag-Leffler function defined by
\begin{align*}
    E_{\alpha,\beta}(z)=\sum_{n=0}^{\infty}\frac{z^n}{\Gamma(\alpha n+\beta)} \qquad \forall z\in\mathbb{C},
\end{align*}
where $\alpha>0$, $\beta\in \mathbb{R}$. We refer to \cite{ML03} and the monographs \cite{Gor20, podlubny99} for a detailed study of this function.

We also recall the asymptotic expansions of the Mittag-Leffler function near $\infty$. We refer to pages 32-35 in \cite{podlubny99} for the proofs.
\begin{lem}
Let $0<\alpha<1$ and $\beta \in \mathbb{R}$ be arbitrary. Let $\mu$ be such that $\frac{\pi \alpha}{2}<\mu<\min \{\pi, \pi \alpha\}$. Then, for any $N\ge 1$,
\begin{equation}\label{asymptotic expansions}
    E_{\alpha, \beta}(z)=-\sum_{k=1}^N\frac{1}{\Gamma(\beta-\alpha k)}\frac{1}{z^k} + \mathcal{O}\left(\frac{1}{|z|^{1+N}} \right), \; |z|\to \infty, \; \mu \leq |\arg (z)| \leq \pi.
\end{equation}
Moreover, there exist positive constants $C$ depending on $(\alpha, \beta, \mu)$ such that
\begin{equation}\label{boundedness M-L}
\left|E_{\alpha, \beta}(z)\right| \leq \frac{C}{1+|z|}, \qquad \mu \leq |\arg (z)| \leq \pi.
\end{equation}
\end{lem}
Consequently, we obtain:
\begin{cor}
\label{est12}
Let $0<\alpha<1$ and $\beta\in \mathbb{R}$. There exists a constant $C_0>0$, depending only on $\alpha$ and $\beta$, such that
\begin{align}\label{es0}
& \left|E_{\alpha, \beta}\left(\mathrm{i} t\right)\right| \leq \frac{C_0}{1+|t|} \leq C_0 \qquad \text{for all } t\in \mathbb{R}.
\end{align}
\end{cor}

\subsection{Fractional differentiation}\label{sec22}
Let $\alpha > 0$ and $p>1$. We define the Riemann-Liouville fractional integral operator by
\begin{equation}
J^{\alpha}v(t):= \frac{1}{\Gamma(\alpha)}\int^t_0 (t-s)^{\alpha-1}
v(s)\, \d s, \qquad v\in L^p(0,T;H).
\end{equation}
Then, one can prove that the operator $J^{\alpha}: L^p(0,T;H) \rightarrow J^{\alpha}L^p(0,T;H)\subset L^p(0,T;H)$ is one-to-one and 
onto. We denote by $J^{-\alpha}:=(J^{\alpha})^{-1}$ its inverse operator. Then, we define the functional space  
\begin{equation}
W_{\alpha,p}(0,T;H) := J^{\alpha}L^p(0,T;H),
\end{equation}
with the corresponding norm $\Vert v\Vert_{W_{\alpha,p}(0,T;H)}:= \Vert J^{-\alpha}v \Vert_{L^p(0,T;H)}$.
We can check that $W_{\alpha,p}(0,T;H)$ is a Banach space with respect to the norm $\Vert \cdot\Vert_{W_{\alpha,p}(0,T;H)}$. Moreover, for the case $p=2,$ we set
$$H_{\alpha}(0,T;H):=W_{\alpha,2}(0,T;H).$$
\begin{prop} \label{fractional_space}
Let $\alpha>\frac{1}{p}$. Then the continuous embedding
$$W_{\alpha,p}(0,T;H) \hookrightarrow \{u\in C([0,T];H) \colon u(0)=0\}$$
holds.
\end{prop}

\begin{proof}
The proof is a modification of \cite[Theorem 2.1.4]{Ya25} concerning the case $p=2$.

Let $u\in L^p(0,T;H)$ and $0 \leq t_0<t \leq T$. We have
\begin{align*}
\Gamma(\alpha)\left(J^\alpha u(t)-J^\alpha u(t_0)\right)= & \int_{t_0}^t(t-s)^{\alpha-1} u(s) \d s \\
& -\int_0^{t_0}\left(\left(t_0-s\right)^{\alpha-1}-(t-s)^{\alpha-1}\right) u(s) \d s=: I_1-I_2.
\end{align*}
Since $\alpha>\frac{1}{p}$, $q(\alpha-1)>-1$. Then, we estimate $I_1$ by Hölder's inequality as follows
\begin{align}
\left\|I_1\right\| & \leq \int_{t_0}^t(t-s)^{\alpha-1}\|u(s)\| \d s \notag\\
& \leq\left(\int_{t_0}^t(t-s)^{q(\alpha-1)} \d s\right)^{\frac{1}{q}}\left(\int_{t_0}^t\|u(s)\|^p \d s\right)^{\frac{1}{p}} \notag\\
&=\left(\frac{\left(t-t_0\right)^{q(\alpha-1)+1}}{q(\alpha-1)+1}\right)^{\frac{1}{q}}\|u\|_{L^p(0, T ; H)}. \label{I1}
\end{align}
Next, we estimate $I_2$. By $\alpha>\frac{1}{p}$, we choose a small $\varepsilon>0$ so that $\alpha>\frac{1}{p}+\varepsilon$. Since $0<1-\alpha<1$, we have $\left|a^{1-\alpha}-b^{1-\alpha}\right| \leq|a-b|^{1-\alpha}$ for all $a, b \geq 0$. Then
$$
\left|(t-s)^{1-\alpha}-\left(t_0-s\right)^{1-\alpha}\right| \leq \left|t-t_0\right|^{1-\alpha}.
$$
Therefore, using $t-s \geq t-t_0$ and $t-s \geq t_0-s$ for $0<s<t_0<t$, we obtain
\begin{align*}
\left\|I_2\right\| & \leq \int_0^{t_0}\left|\left(t_0-s\right)^{\alpha-1}-(t-s)^{\alpha-1}\right|\|u(s)\| \d s \\
& =\int_0^{t_0} \frac{\left|(t-s)^{1-\alpha}-\left(t_0-s\right)^{1-\alpha}\right|}{(t-s)^{1-\alpha}\left(t_0-s\right)^{1-\alpha}}\|u(s)\| \d s \\
& \leq \int_0^{t_0} \frac{\left(t-t_0\right)^{1-\alpha}}{(t-s)^{1-\alpha-\varepsilon}(t-s)^{\varepsilon}\left(t_0-s\right)^{1-\alpha}}\|u(s)\| \d s \\
& \leq \frac{\left(t-t_0\right)^{1-\alpha}}{\left(t-t_0\right)^{1-\alpha-\varepsilon}} \int_0^{t_0}\left(t_0-s\right)^{\alpha-1-\varepsilon}\|u(s)\| \d s\\
& \leq\left(t-t_0\right)^{\varepsilon}\left(\int_0^{t_0}\left(t_0-s\right)^{q(\alpha-1-\varepsilon)} \d s\right)^{\frac{1}{q}}\left(\int_0^{t_0}\|u(s)\|^p \d s\right)^{\frac{1}{p}} \\
& \leq\left(t-t_0\right)^{\varepsilon}\left(\frac{t_0^{q(\alpha-1-\varepsilon)+1}}{q(\alpha-1-\varepsilon)+1}\right)^{\frac{1}{q}}\|u\|_{L^p(0, T; H)},
\end{align*}
where $q(\alpha-1-\varepsilon)+1>0$ because $\alpha>\frac{1}{p}+\varepsilon$. Hence,
$$
\left\|\Gamma(\alpha)\left(J^\alpha u(t)-J^\alpha u\left(t_0\right)\right)\right\| \leq C\left(\left(t-t_0\right)^{\alpha-\frac{1}{p}}+\left(t-t_0\right)^{\varepsilon} T^{\alpha-\frac{1}{p}-\varepsilon}\right)\|u\|_{L^p(0, T ; H)},
$$
This implies that $J^\alpha u \in C([0, T] ; H)$.
Moreover, by \eqref{I1}, we have
\begin{align*}
\left\|\Gamma(\alpha) J^\alpha u(t)\right\| & \leq \left(\frac{t^{q(\alpha-1)+1}}{q(\alpha-1)+1}\right)^{\frac{1}{q}} \|u\|_{L^p(0, T ; H)} \leq C t^{\alpha-\frac{1}{p}}\|u\|_{L^p(0, T ; H)}.
\end{align*}
Therefore, $\lim\limits_{t \downarrow 0}\left\|J^\alpha u(t)\right\|=0$. This completes the proof.
\end{proof}

Next, we define the time-fractional derivative operator in $W_{\alpha,p}(0,T;H)$ by 
\begin{equation}\label{deffd}
\partial_t^\alpha := J^{-\alpha}, \qquad D(\partial_t^\alpha) = W_{\alpha,p}(0,T;H).
\end{equation}
We can prove that $\partial_t^\alpha : D(\partial_t^\alpha) = W_{\alpha,p}(0,T;H)\, \rightarrow
\, L^p(0,T;H)$ is a closed operator. The notation $(\partial_t^\alpha)'$ stands for the dual operator of $\partial_t^\alpha : W_{\alpha,p}(0,T;H) \rightarrow L^p(0,T;H)$.

For $1 \le r \le \infty$, we define an operator $\tau_T: L^r(0,T;H) \rightarrow L^r(0,T;H)$ by 
$(\tau_T v)(t) = v(T-t)$.
Then, we can prove (see \cite{Ya18} for the proof):
\begin{lem}
Let $1 < p < \infty$. Then, 
$$
D((\partial_t^\alpha)') = \tau_T W_{\alpha,q}(0,T;H), \quad
(\partial_t^\alpha)'u = \tau_T\partial_t^\alpha (\tau_T u).
$$
\end{lem}
By the definition of the dual operator $(\partial_t^\alpha)'$, we see
$$
\langle u, \, (\partial_t^\alpha)'v \rangle 
= \langle \partial_t^\alpha u,\, v \rangle \quad \mbox{for } u\in D(\partial_t^\alpha) 
\mbox{ and }v \in D((\partial_t^\alpha)').
$$
We define the space of test functions by
$$
\Psi:= \{ \psi\in C^{\infty}([0,T]; D(A))\colon \; 
\psi(T) = 0  \}.
$$
Then one can prove that $\Psi \subset D((\partial_t^\alpha)').$

\section{Notion of weak solution} \label{sec4}

First, to prove the uniqueness of solutions, we need the following coercivity result.
\begin{lem}\label{coer}
Let $v\in  H_{\alpha}(0,T;\mathbb{C})$. Then
$$\mathrm{Re} \int_0^T \overline{v}(t) \partial_t^\alpha v(t)\,\d t \ge \frac{T^{-\alpha}}{2\Gamma(1-\alpha)} \|v\|^2_{L^2(0,T)},$$
where $\overline{v}$ denotes the conjugate function of $v$.
\end{lem}
\begin{proof}
The proof is a consequence of the analogue for real functions. Indeed, let $v=u+\mathrm{i} w$ with real-valued $u, w \in H_{\alpha}(0,T;\mathbb{R})$. Linearity gives $\partial_t^\alpha v=\partial_t^\alpha u+\mathrm{i} \partial_t^\alpha w$ and $\partial_t^\alpha u, \partial_t^\alpha w$ are real-valued. Hence
$$
\overline{v} \partial_t^\alpha v=u \partial_t^\alpha u+w \partial_t^\alpha w+\mathrm{i}\left(u \partial_t^\alpha w-w \partial_t^\alpha u\right).
$$
Taking the real part,
$$
\mathrm{Re} \int_0^T \overline{v} \partial_t^\alpha v\, \d t=\int_0^T\left(u \partial_t^\alpha u+w \partial_t^\alpha w\right) \d t.
$$
By the coercivity of real-valued functions (see \cite[Theorem 3.1~(ii)]{KRY20}), we have
$$
\int_0^T \phi(t) \partial_t^\alpha \phi(t)\, \d t \geq \frac{T^{-\alpha}}{2\Gamma(1-\alpha)}\|\phi\|_{L^2(0, T)}^2
$$
for $\phi=u$ and $\phi=w$. Then
$$
\mathrm{Re} \int_0^T \overline{v} \partial_t^\alpha v\, \d t \geq \frac{T^{-\alpha}}{2\Gamma(1-\alpha)}\left(\|u\|_{L^2(0,T)}^2+\| w \|_{L^2(0,T)}^2\right)=\frac{T^{-\alpha}}{2\Gamma(1-\alpha)}\|v\|_{L^2(0,T)}^2 .
$$
\end{proof}

Let $0<\alpha<1$ and consider the fractional Schrödinger equation
\begin{equation}\label{fseq}
\partial_t^\alpha u -\mathrm{i} Au=f, \qquad t \in (0,T).
\end{equation}
For $t\in[0,T]$, we decompose the solution
\begin{equation*}
u(t)=\sum_{n=1}^\infty u_n(t)\,\varphi_n,
\end{equation*}
and set $f_n(t):=\langle f(t),\varphi_n\rangle$.

Taking the inner product of \eqref{5fseq} with $\varphi_n$ and using the symmetry of $A$, we formally obtain
\begin{equation}\label{odef}
\partial^{\alpha}_{t} u_n(t)+\mathrm{i}\lambda_n u_n(t)=f_n(t), \qquad t\in (0,T).
\end{equation}
We prove the uniqueness of a solution to \eqref{odef} in $H_{\alpha}(0,T;\mathbb{C})$ as follows. Let $v\in H_{\alpha}(0,T;\mathbb{C})$ such that $\partial_t^\alpha v(t) + \mathrm{i} \lambda_{n}v(t) = 0$ for all $t\in(0,T)$. Multiplying by $\overline{v},$ integrating on $(0,T)$ and taking the real part, Lemma \ref{coer} yields
$$0=\mathrm{Re} \int_0^T \overline{v}(t) \partial_t^\alpha v(t)\, \d t \ge \frac{T^{-\alpha}}{2\Gamma(1-\alpha)} \|v\|^2_{L^2(0,T)}.$$
Hence, $v=0$ in $L^2(0,T)$.

Therefore, we see that the solution to \eqref{odef} is unique in $(0,T)$, and we can verify that
$$u_n(t):=\int_0^t f_n(s)\,(t-s)^{\alpha-1}E_{\alpha,\alpha}\big(-\mathrm{i}\lambda_n (t-s)^\alpha\big)\,\d s.$$
Therefore, we obtain
\begin{equation}\label{solf}
u(t)=\sum_{n=1}^\infty \left[\int_0^t \langle f(s),\varphi_n\rangle\,(t-s)^{\alpha-1}E_{\alpha,\alpha}\big(-\mathrm{i}\lambda_n (t-s)^\alpha\big)\,\d s \right] \varphi_n.
\end{equation}

\begin{defn}
For $f\in L^p(0,T;H)$, we call $u$ a weak solution to equation $\eqref{fseq}$ if $u\in L^p(0,T;H)$ and satisfies
\begin{equation}\label{weaks}
 \, _{L^p(0,T;H)}\langle u, \, ((\partial_t^\alpha)' +\mathrm{i} A)\psi \rangle_{L^{q}(0,T;H)}
= \, _{L^p(0,T;H)}\langle f,\, \psi \rangle_{L^{q}(0,T;H)}
\end{equation}
for all $\psi \in \Psi$, where $A\psi$ denotes the function $t\mapsto A\psi(t)$.
\end{defn}

For $s\in\mathbb{R}$, we define the diagonal operator $K(s): H\to H$ by
\begin{equation}\label{dfks}
K(s)\varphi_n = k_n(s)\varphi_n \qquad \text{for all } n\in \mathbb{N},
\end{equation}
where $$k_n(s)=\begin{cases}
    s^{\alpha-1} E_{\alpha,\alpha}(-\mathrm{i}\lambda_n s^\alpha),\quad &s>0,\\
    0,\quad &s\leq 0.
\end{cases} $$
Next, we collect some properties of the operator $K(s)$. 
\begin{lem}\label{lmks} For each $s>0$, the following properties hold:
\begin{itemize}
    \item[(i)] The operator $K(s)\in\mathcal{B}(H)$ and 
\begin{equation}\label{estks}
\|K(s)\|_{\mathcal{B}(H)}\le C_0 s^{\alpha-1}.
\end{equation}
\item[(ii)] $K(s)H\subset D(A)$, the operator $AK(s)\in \mathcal{B}(H)$ and
\begin{equation}
\|AK(s)\|_{\mathcal{B}(H)}\le C_0 s^{-1}.
\end{equation}
\item[(iii)] $K(s)$ and $A$ commute on $D(A)$.
\end{itemize}
\end{lem}
\begin{proof}
(i) For $s>0$ and $x=\displaystyle\sum_{n=1}^\infty \langle x,\varphi_n\rangle \varphi_n \in H$, the estimate \eqref{es0} yields
\begin{align*}
\|K(s)x\|^2 &= \sum_{n=1}^\infty |k_n(s)|^2 |\langle x,\varphi_n\rangle|^2\\
& \leq C_0^2 s^{2\alpha-2}\|x\|^2.
\end{align*}
(ii) Moreover, by \eqref{es0}, we obtain
\begin{align*}
\|AK(s)x\|^2 &= \sum_{n=1}^\infty \lambda_n^2 |k_n(s)|^2 |\langle x,\varphi_n\rangle|^2\\
&\le C_0^2 s^{-2} \sum_{n=1}^\infty\frac{\lambda_n^2 s^{2\alpha}}{1+\lambda_n^2 s^{2\alpha}}|\langle x,\varphi_n\rangle|^2\\
& \leq C_0^2 s^{-2}\|x\|^2.
\end{align*}
(iii) Since $K(s)$ and $A$ are diagonal, then they commute on $D(A)$.
\end{proof}

The next lemma provides a representation of the formal solution \eqref{solf} as a convolution with an operator-valued kernel.
\begin{lem}\label{lm:rep}
    Let $f\in L^p(0,T;H)$ for some $p\ge 1$. Then the formal solution \eqref{solf} can be represented as follows
    \begin{equation}\label{representation}
        u(t)=\int_{0}^{t}K(t-s)f(s)~\d s.
    \end{equation}
\end{lem}
\begin{proof}
First, using \eqref{es0}, we obtain
\begin{align}
    |k_n(\tau)|\leq C_{0}\tau^{\alpha-1} \qquad \text{for all } \tau>0. \label{est:ker}
\end{align} 
We now justify interchanging the sum and the integral. Let $N\in\mathbb{N}$ and $t\in (0,T)$. Put
$$S_N(t) := \sum_{n=1}^N \left( \int_0^t f_n(s) k_n(t-s) \, \d s \right) \varphi_n = \int_0^t g_N(s) \, \d s,$$ where $f_n(s):=\langle f(s),\varphi_n\rangle$ and $g_N(s) := \displaystyle\sum_{n=1}^N f_n(s) k_n(t-s) \varphi_n$.  
Using \eqref{est:ker}, for a.e.\ $s\in(0,t)$, we obtain
\begin{align*}
	\|g_N(s)\|&=\left(\sum_{n=1}^N |f_n(s)|^2 |k_n(t-s)|^2 \right)^{\frac{1}{2}}\\
	&\leq C_0 (t-s)^{\alpha-1}\|f(s)\|.
\end{align*}
Since $\alpha\in (0,1)$ and $f\in L^p(0,T;H)$, by the properties of convolution, we have
\begin{align*}
    \int_{0}^{t}(t-s)^{\alpha-1}\|f(s)\|~\d s <\infty\qquad \mbox{for a.e.}\; t\in (0,T).
\end{align*}
Hence, the dominated convergence theorem for Bochner integrals implies
\begin{align*}
    u(t)=& \lim_{N\to\infty} \int_0^t g_N(s)\, \d s\\
    =& \int_0^t \sum_{n=1}^\infty f_n(s) k_n(t-s)\, \varphi_n \, \d s\\
    =& \int_0^t K(t-s)f(s)\, \d s.
\end{align*}
\end{proof}

Hereafter, we set
$$u(g)(t)=\int_{0}^{t}K(t-s)g(s)\d s, \qquad g\in L^p(0,T;H),\quad 0<t<T.$$
\begin{lem}\label{strons}
Let $g \in C_c^{\infty}(0, T;D(A))$. Then $u(g) \in W_{\alpha,p}(0,T;H)\cap L^p(0,T;D(A))$.
\end{lem}
\begin{proof}
For $\varepsilon>0$ sufficiently small, we set
$$u_\varepsilon(g)(t)=\int_{0}^{t-\varepsilon}K(t-s)g(s)\d s, \quad 0<t<T.$$
Therefore, by Riemann-sum approximation and Lemma \ref{lmks}-(iii), we can prove $u_\varepsilon(g)(t)\in D(A)$ and $Au_\varepsilon(g)(t)=\displaystyle\int_{0}^{t-\varepsilon}K(t-s)Ag(s)\d s$. Therefore, we have $u_\varepsilon(g)(t) \to u(g)(t)$ and $Au_\varepsilon(g)(t) \to \displaystyle\int_{0}^{t}K(t-s)Ag(s)\d s$ as $\varepsilon \to 0$ in $H$. Since $A$ is closed, we obtain $u(g)(t)\in D(A)$ and
$$Au(g)(t)=A\int_{0}^{t}K(t-s)g(s)\d s=\int_{0}^{t}K(t-s)Ag(s)\d s.$$
Then, by Lemma \ref{lmks}-(i),
\begin{align*}
\|Au(g)(t)\| &\le \int_0^{t} \|K(t-s)\|_{\mathcal{B}(H)}\|Ag(s)\|\,\d s\\
&\le C_0 \int_0^t (t-s)^{\alpha-1}\|Ag(s)\|\,\d s.
\end{align*}
Young's convolution inequality yields
\begin{align}
\|u(g)\|_{L^p(0,T;D(A))} &\le C_0 \|t^{\alpha-1}\|_{L^1(0,T)} \|g\|_{L^p(0,T;D(A))} \nonumber\\
& \le C_0 \frac{T^\alpha}{\alpha} \|g\|_{L^p(0,T;D(A))} <\infty. \label{Weak_estimate}
\end{align}
Thus $u(g)\in L^p(0,T;D(A))$.

Using the Laplace transform, we can prove that
$$\partial_t^\alpha u(g) =\mathrm{i} Au(g)+g.$$
Therefore, $\partial_t^\alpha u(g) \in L^p(0,T;H)$, and hence, $u(g)=J^\alpha(\partial_t^\alpha u(g)) \in W_{\alpha,p}(0,T;H)$.
\end{proof}

\begin{theorem}
Assume $f\in L^p(0,T;H).$ Then there exists a unique weak solution to $\eqref{fseq}$ given by \eqref{solf}.
\end{theorem}
\begin{proof}
First, we prove the uniqueness. Let $u\in L^p(0,T;H)$ such that
\begin{equation} \label{deq}
 \, _{L^p(0,T;H)}\langle u, \, ((\partial_t^\alpha)' +\mathrm{i} A)\psi \rangle_{L^{q}(0,T;H)} 
=0
\end{equation}
for all $\psi \in \Psi$. We need to prove that $u=0$. We write $u(t)$ as
$$u(t)=\sum_{n=1}^\infty u_n(t)\,\varphi_n,\qquad u_n\in L^p(0,T;\mathbb{C}),$$
where $u_n(t)=\langle u(t),\varphi_n\rangle$. For a fixed $n\ge 1$, we take arbitrary function $\eta\in C^\infty([0,T];\mathbb{C})$ such that $\eta(T)=0$, and we set $\psi(t)=\eta(t)\varphi_n\in\Psi$. Then, by \eqref{deq}, we obtain
$$ 0=\int_0^T u_n(t)\big((\partial_t^\alpha)'\overline{\eta}(t)+\mathrm{i}\lambda_n\overline{\eta}(t)\big)\,\d t.$$
It remains to prove that
$$\Phi_n:=\left\{(\partial_t^\alpha)'\overline{\eta} +\mathrm{i}\lambda_n\overline{\eta}:\ \overline{\eta}\in C^\infty([0,T];\mathbb{C}),\ \overline{\eta}(T)=0\right\}$$
is dense in $L^q(0,T;\mathbb{C})$. Next, we prove that $C_c^\infty(0,T;\mathbb{C}) \subset \Phi_n$. For $g\in C_c^\infty(0,T;\mathbb{C})$, we solve the fractional ordinary equation
\begin{equation}\label{ode}
(\partial_t^\alpha)'\overline{\eta} +\mathrm{i}\lambda_n\overline{\eta} = g.
\end{equation}
By \cite[Lemma 3.1]{Ya18}, we have $(\partial_t^\alpha)'\overline{\eta}$ is the right-sided Riemann–Liouville derivative for $\tau_T \overline{\eta} \in D(\partial_t^\alpha)$. Then, the solution to \eqref{ode} is given by
$$\overline{\eta}(t):=\int_t^T (s-t)^{\alpha-1}E_{\alpha,\alpha}\big(-\mathrm{i}\lambda_n (s-t)^\alpha\big)\,g(s)\,\d s.$$
Therefore, $\Phi_n$ is dense in $L^q(0,T;\mathbb{C})$ and $u_n=0$ for all $n\ge 1$. Thus $u=0$.

Next, we prove that the formal solution \eqref{solf} is a weak solution to $\eqref{fseq}$.

By Lemma \ref{lm:rep}, the formal solution \eqref{solf} can be written as
$$u(t)=\int_{0}^{t}K(t-s)f(s)~\d s.$$
Then
\begin{align*}
\|u(t)\| &\le  \int_0^t \|K(t-s)\|_{\mathcal{B}(H)}\|f(s)\|\,\d s\\
&\le C_0 \int_0^t (t-s)^{\alpha-1}\|f(s)\|\,\d s.
\end{align*}
By Young's convolution inequality, we infer
\begin{align*}
\|u\|_{L^p(0,T;H)} &\le C_0 \|t^{\alpha-1}\|_{L^1(0,T)} \|f\|_{L^p(0,T;H)}\\
& \le C_0 \frac{T^\alpha}{\alpha} \|f\|_{L^p(0,T;H)} <\infty.
\end{align*}
Then $u\in L^p(0,T;H)$.

Now, we choose a sequence $f_n \in C_c^{\infty}(0, T;D(A))$, $n \in \mathbb{N}$ such that $f_n \to f$ in $L^p(0, T;H)$. By Lemma \ref{strons}, we have $u(f_n) \in D(\partial_t^\alpha) \cap L^p(0,T;D(A))$. Then we obtain
\begin{equation}\label{dualfn}
\left\langle u(f_n),\left(\left(\partial_t^\alpha\right)^{\prime}+\mathrm{i}A\right) \psi\right\rangle=\left\langle f_n, \psi\right\rangle \quad \text{ for all } \psi \in \Psi \text{ and } n \in \mathbb{N}.
\end{equation}
On the other hand, one has
$$
u(f_n) \to u(f) \quad\text{ in } L^p(0,T;H) \quad \text{ as } n \rightarrow \infty .
$$
Letting $n \rightarrow \infty$ in \eqref{dualfn}, we obtain
$$
\left\langle u(f),\left(\left(\partial_t^\alpha\right)^{\prime}+\mathrm{i}A\right) \psi\right\rangle=\left\langle f, \psi\right\rangle \quad \text{ for all } \psi \in \Psi.
$$
Therefore, $u$ is the unique weak solution to $\eqref{fseq}$.
\end{proof}

Regarding the continuity of the weak solution on $[0,T]$ with values in $H$, we establish the following result.
\begin{prop}\label{propc}
Assume $f\in L^p(0,T;H)$ for $p>\frac{1}{\alpha}$. Then the weak solution to \eqref{fseq} satisfies $u\in C([0,T];H)$, and there exists a constant $C=C(\alpha,p) >0$ such that
\begin{align}
    \|u\|_{C([0,T];H)}\le CT^{\alpha-\frac{1}{p}}\,\|f\|_{L^p(0,T;H)}. \label{energy:conti}
\end{align}
\end{prop}
\begin{proof}
Recall from Lemma \ref{lmks}-(i) that we have
\[\|K(s)\|_{\mathcal{B}(H)}\leq C_{0}s^{\alpha-1} \qquad \text{for all } s>0.\]
Since $\alpha p> 1$, then $(\alpha-1)q>-1$ and 
\begin{align*}
    \|K\|_{L^{q}(0,T;\mathcal{B}(H))}\leq& C_{0}\left(\frac{T^{(\alpha-1)q+1}}{(\alpha-1)q+1}\right)^{\frac{1}{q}}\\
    \leq & CT^{\alpha-\frac{1}{p}}.
\end{align*}
where $C=C(\alpha,p)>0$.
For \(t\in[0,T]\), using \eqref{representation} and Hölder's inequality, we get
\begin{align}
	\|u(t)\| &\le  \int_0^t \|K(s)\|_{\mathcal{B}(H)}\|f(s)\|\,\d s \nonumber\\
	&\le \|K\|_{L^{q}(0,T;\mathcal{B}(H))} \|f\|_{L^{p}(0,T;H)} \nonumber\\
	&\le CT^{\alpha-\frac{1}{p}}\|f\|_{L^{p}(0,T;H)}. \label{est:conti}
\end{align}
Thus \(u(t)\) is well-defined for all \(t\in[0,T]\).
%and the estimate holds.
	To prove the right continuity at \(t\in [0,T)\), let \(h\geq 0\) be small so \(t+h\in[0,T]\) (left continuity can be proved similarly). By a change of variables, we can write
	\begin{align*}
		u(t+h)-u(t)
		%=& \int_0^{t} \big(K(s+h)-K(s)\big) f(t-s)\,\d s
		%+ \int_{0}^{h} K(s)\,f(t+h-s)\,\d s\\
		=& \int_0^{t} \big(\tau_hK(s)-K(s)\big) f(t-s)\,\d s
		+ \int_{0}^{h} K(s)\,\tau_hf(t-s)\,\d s.
	\end{align*}
	Taking the \(H\)-norm and using Hölder's inequality yields
	\[
	\|u(t+h)-u(t)\|
	\le \|f\|_{L^p(0,T;H)} \Big( \|\tau_h K-K\|_{L^q(0,T;\mathcal{B}(H))}
	+ \|K\|_{L^q(0,h;\mathcal{B}(H))}\Big).
	\]
	The second term tends to \(0\) as \(h\to0\) by the integrability of $K\in L^q(0,T;\mathcal{B}(H))$.
    
	Next, we prove that the first term tends to $0$ as $h\to 0$. Let $\varepsilon>0$. Since continuous compactly supported functions are dense in $L^q(0, T;\mathcal{B}(H))$, choose $\Psi \in C_c((0, T) ; \mathcal{B}(H))$ with
	$$
	\|K-\Psi\|_{L^q(0, T;\mathcal{B}(H))}<\frac{\varepsilon}{3}.
	$$
	%Because $\Psi$ has compact support in $(0, T)$, there exists $\delta>0$ such that $\operatorname{supp} \Psi \subset[\delta, T-\delta]$. For $|h|<\delta$ both $\Psi(s)$ and $\Psi(s+h)$ are defined and, 
    By uniform continuity of $\Psi$ on its compact support, we can choose $|h|$ small enough so that
	$$
	\|\tau_h\Psi-\Psi\|_{L^q(0, T;\mathcal{B}(H))}<\frac{\varepsilon}{3}.
	$$
	Now, for such $h$, the triangle inequality gives
	\begin{align*}
		\|\tau_h K-K\|_{L^q(0, T;\mathcal{B}(H))}\le& \|\tau_h K-\tau_h \Psi\|_{L^q(0, T;\mathcal{B}(H))} + \|\tau_h \Psi -\Psi\|_{L^q(0, T;\mathcal{B}(H))}\\
		&+ \|\Psi -K\|_{L^q(0, T;\mathcal{B}(H))}.
	\end{align*}
	Each of the terms is bounded by $\frac{\varepsilon}{3}$. This proves the desired result. 
    Finally, $u\in C([0,T];H)$, and the estimate \eqref{energy:conti} follows from \eqref{est:conti}.
\end{proof}

\section{Maximal regularity}\label{sec5}
In this section, we investigate the maximal regularity for the following fractional Schr\"odinger equation
\begin{equation}\label{5fseq}
\partial_t^\alpha u -\mathrm{i} Au=f, \qquad t \in (0,T).
\end{equation}
\begin{defn}
Let $1<p<\infty$. We say that the equation $\eqref{5fseq}$ satisfies the maximal $L^p$-regularity, if for all $f\in L^p(0,T;H)$ the unique weak solution $u$ belongs to $W_{\alpha,p}(0,T;H)$. That is, the solution to the equation $\eqref{5fseq}$ belongs to $W_{\alpha,p}(0,T;H)\cap L^p(0,T;D(A))$, where $D(A)$ is endowed with the graph norm.
\end{defn}
In this case, by the Closed Graph Theorem, there is $C>0$ such that
\begin{equation}
\|u\|_{W_{\alpha,p}(0,T;H)}+\|Au\|_{L^p(0,T;H)} \le C \|f\|_{L^p(0,T;H)}.
\end{equation}

\subsection{Maximal $L^2$-regularity}\label{sec51}
We start with the $L^2$-case, where we prove the maximal $L^2$-regularity in an elementary manner.
\begin{theorem}\label{thmMR}
Assume $f\in L^2(0,T;H).$ Then \eqref{5fseq} satisfies the maximal $L^2$-regularity. In particular, $u\in L^2(0,T;D(A))$ and there exists a constant $C>0$ such that
\[
\|u\|_{H_{\alpha}(0,T;H)}+\|u\|_{L^2(0,T;D(A))}\le C\,\|f\|_{L^2(0,T;H)}.
\]
\end{theorem}

\begin{proof}
For $t\in[0,T]$, write
\begin{equation}\label{eq:y-def}
u(t)=\sum_{n=1}^\infty u_n(t)\,\varphi_n,\qquad
u_n(t):=\int_0^t f_n(s)\,(t-s)^{\alpha-1}E_{\alpha,\alpha}\big(-\mathrm{i}\lambda_n (t-s)^\alpha\big)\,\d s,
\end{equation}
where $f_n(s):=\langle f(s),\varphi_n\rangle$.

We can check that
\begin{equation}\label{four}
\partial^{\alpha}_{t} u_n(t)+\mathrm{i}\lambda_n u_n(t)=f_n(t).
\end{equation}
Using Young's convolution inequality, we obtain
\begin{align}
\|\partial^{\alpha}_{t} u_n(t)\|^2_{L^2(0,T)} & \leq C \int_0^T\left|f_n(t)\right|^2 d t + C\left(\int_0^T\left|f_n(t)\right|^2 \d t\right)\left(\int_0^T\left|\lambda_n t^{\alpha-1} E_{\alpha, \alpha}\left(-\mathrm{i}\lambda_n t^\alpha\right)\right| \d t\right)^2.
\end{align}
Fix $n\in \mathbb N$. By the change of variable $s=\lambda_n^{\frac{1}{\alpha}} t$, we have
\begin{align*}
\int_0^T \lambda_n t^{\alpha-1} \left|E_{\alpha, \alpha}\left(-\mathrm{i}\lambda_n t^\alpha\right)\right| \d t &= \int_0^{\lambda_n^{\frac{1}{\alpha}}T} s^{\alpha-1} \left|E_{\alpha, \alpha}\left(-\mathrm{i}s^\alpha\right)\right| \d s\\
&\leq \int_0^{\infty} s^{\alpha-1}\left|E_{\alpha, \alpha}\left(-\mathrm{i}s^\alpha\right)\right| \d s =:I(\alpha).
\end{align*}
Note that the last integral is convergent, i.e., $I(\alpha)<\infty$. Indeed, near $0$, by \eqref{es0}, we have $\left|E_{\alpha, \alpha}\left(-\mathrm{i}s^\alpha\right)\right| \leq C$ for all $s>0$, and
$\displaystyle\int_0^{1} s^{\alpha-1} \d s=\frac{1}{\alpha}$. On the other hand, near $\infty$, the asymptotic expansion \eqref{asymptotic expansions} at order $N=2$ yields $E_{\alpha, \alpha}\left(-\mathrm{i}s^\alpha\right)=\mathcal{O}(s^{-2\alpha})$ as $s\to \infty$. Then $s^{\alpha-1} E_{\alpha, \alpha}\left(-\mathrm{i}s^\alpha\right)=\mathcal{O}(s^{-\alpha-1})$ and $\displaystyle\int_{1}^{\infty} s^{-\alpha-1} \d s=\frac{1}{\alpha}$.

Therefore, 
\begin{equation}\label{dta}
\|\partial^{\alpha}_{t} u_n(t)\|^2_{L^2(0,T)} \leq C \int_0^T\left|f_n(t)\right|^2 \d t.
\end{equation}
Hence,
\begin{align*}
\left\|\partial^{\alpha}_{t} u\right\|_{L^2(0,T;H)}^2 =\sum_{n=0}^{\infty} \|\partial^{\alpha}_{t} u_n(t)\|^2_{L^2(0,T)} \le C \sum_{n=0}^{\infty} \int_0^T |f_n(t)|^2 \d t = C \|f\|_{L^2(0,T;H)}^2.
\end{align*}
Moreover, by \eqref{four} and \eqref{dta}, we infer
$$\|\lambda_n u_n(t)\|^2_{L^2(0,T)} \leq C \int_0^T\left|f_n(t)\right|^2 \d t.$$
Then, similarly, we obtain
$$\left\|A u\right\|_{L^2(0,T;H)} \le C \|f\|_{L^2(0, T;H)}.$$
\end{proof}
\begin{rmk}
The proof of Theorem \ref{thmMR} breaks down when $\alpha=1$ because the constant $I(\alpha)$ becomes infinite.
\end{rmk}
\begin{rmk}
The analog of Theorem \ref{thmMR} has been proven in \cite[Theorem 2.2~(i)]{SY11} for fractional diffusion equations. The proof relies on the complete monotonicity of the function $E_{\alpha,1}(-s^\alpha)$ for $s > 0$ when $0<\alpha<1$; see, e.g., \cite{Po48} and also \cite{Sc96}. In particular, this implies that $E_{\alpha,\alpha}(-s^\alpha)\ge 0$ for $s \ge 0$. In our case, the complete monotonicity of the function $|E_{\alpha,1}(-\mathrm{i}s^\alpha)|$, for $s > 0$, seems to be difficult; we refer to Conjecture 3 in \cite{CEMY25} for more details. Thus, we needed to modify the proof without using the complete monotonicity.
\end{rmk}

\subsection{Maximal $L^p$-regularity}\label{sec52}
Concerning maximal $L^p$-regularity, we first prove the following partial result.
\begin{theorem}\label{bthmMRL^p}
Assume $f\in L^p(0,T;D(A)).$ Then \eqref{5fseq} satisfies the maximal $L^p$-regularity. In particular, $u\in L^p(0,T;D(A))$ and there exists a constant $C>0$ such that
\[
 \|u\|_{L^p(0,T;D(A))}\le C\,\|f\|_{L^p(0,T;D(A))}.
\]
\end{theorem}
\begin{proof}
Let $f \in L^{p}(0,T;D(A))$ be arbitrary. Choose a sequence of functions $f_{n} \in C_c^{\infty}(0, T;D(A))$ 
converging to $f$ in $L^{p}(0,T; D(A))$, and denote the unique weak solution 
associated with $f_{n}$ by $u_{n}$. By Lemma \ref{strons}, we have $u_n\in L^p(0,T;D(A))$.
By the estimate \eqref{Weak_estimate}, the functions $\{A u_{n}\}$ form a Cauchy sequence in $L^{p}(0,T;H)$ and therefore converges to a 
limit $v$ in $L^{p}(0,T;H)$. Using the definition of weak solutions, we can prove that \(\{u_{n}\}\) converges 
weakly to \(u\) in \(L^{p}(0,T;H)\), where \(u\) is the weak solution of 
\eqref{5fseq} associated with \(f\). Hence,
\[
u_n\rightharpoonup u\quad\text{weakly in }L^p(0,T;H),
\qquad Au_n\to v\quad\text{strongly in }L^p(0,T;H).
\]
Then 
\[
(u_n, Au_n)\rightharpoonup (u,v)\quad\text{weakly in }L^p(0,T;H\times H).
\]
On the other hand, since \(\{(u_n,Au_n)\}\subset L^p(0,T; \mathcal{G}(A))\), where 
\(\mathcal{G}(A)\subset H\times H\) denotes the graph of \(A\), and since \(A\) is closed, 
the graph \(\mathcal{G}(A)\) is a closed subset of \(H\times H\). 
It follows that the set $L^p(0,T;\mathcal{G}(A))$
is a closed subspace of \(L^p(0,T;H\times H)\) (due to Rellich–Kondrachov). 
In particular, being a closed linear subspace in a Banach space, 
\(L^p(0,T;\mathcal{G}(A))\) is weakly closed in \(L^p(0,T;H\times H)\) (see Theorem 3.7 in \cite{Bre2011}). Consequently
\((u,v)\in L^p(0,T;\mathcal{G}(A))\).
Then $u\in L^p(0,T;D(A))$ and $Au=v$ in $L^p(0,T;H)$. Consequently 
\[
u_n \rightharpoonup u\quad\text{weakly in }L^p(0,T;D(A)).
\]
Then, $\{\|u_n \|\}$ is bounded in $L^p(0,T;D(A))$ and (see Proposition 3.5 in \cite{Bre2011})
\begin{align*}
    \|u \|_{L^p(0,T;D(A))} &\le \liminf \|u_n \|_{L^p(0,T;D(A))}.
\end{align*}
Using \eqref{Weak_estimate}, it follows 
\begin{align*}
    \|u \|_{L^p(0,T;D(A))} &\le \liminf \,(C_0 \frac{T^\alpha}{\alpha} \|f_n\|_{L^p(0,T;D(A))})\\
    &= C_0 \frac{T^\alpha}{\alpha} \|f\|_{L^p(0,T;D(A))}.
\end{align*}
\end{proof}

To prove the maximal $L^p$-regularity for $1<p<\infty$, we use the following theorem; see \cite[Corollary 16]{Schw61}. Let $\mathcal{S}(\mathbb{R};H)$ denote the Schwartz space of rapidly decreasing smooth $H$-valued functions.
\begin{theorem}[Mikhlin's multiplier theorem]\label{mikh}
Assume that for a function $m \in C^1(\mathbb{R}^*;\mathcal{B}(H))$ the sets
$$
\{m(s): s \in \mathbb{R}^*\} \text { and }\left\{s\, m^{\prime}(s): s \in \mathbb{R}^*\right\}
$$
are bounded in $\mathcal{B}(H)$. Then the Fourier multiplier operator
$$
\mathcal{T}_m f=\mathcal{F}^{-1}(m(\cdot) \widehat{f}(\cdot)), \quad f \in \mathcal{S}(\mathbb{R};H),\quad
$$
extends to a bounded operator $\mathcal{T}_m$ on $L^p(\mathbb{R};H)$ for $1<p<\infty$. Here, $\widehat{f}$ denotes the Fourier transform of $f\in L^1(\mathbb{R};H)$ defined by
$$\widehat{f}(s)=\int_{\mathbb{R}} e^{-\mathrm{i}st}\,f(t)\,\d t, \qquad s\in \mathbb{R}.$$
\end{theorem}
\begin{theorem}\label{thmMRL^p}
Let $0<\alpha<1,$ $1<p<\infty$, and assume that $f\in L^p(0,T;H).$ Then \eqref{5fseq} satisfies the maximal $L^p$-regularity. In particular, $u\in L^p(0,T;D(A))$ and there exists a constant $C>0$ such that
\[
\|u\|_{W_{\alpha,p}(0,T;H)}+\|u\|_{L^p(0,T;D(A))}\le C\,\|f\|_{L^p(0,T;H)}.
\]
\end{theorem}
\begin{proof}
Note that, for $g \in C_c^{\infty}(0, T;D(A))$, 
$$Au(g)(t)=\int_{0}^{t}AK(t-s)g(s)\d s$$
is well-defined (see Lemma \ref{strons}). Now, for $g\in \mathcal{S}(\mathbb{R};H)$, $Au(g)(t)$ is defined as a singular integral of convolution type, since the kernel $AK(s)$ behaves like $s^{-1}$; see Lemma \ref{lmks}-(ii).\\
The Fourier transform is given by
$$\widehat{Au(g)}(s)=\widehat{AK(t)}(s)\widehat{g}(s), \qquad g\in \mathcal{S}(\mathbb{R};H),\quad s\in \mathbb{R}.$$
We have
$$
AK(t)\,\varphi_n =-\lambda_n k_n(t)\,\varphi_n. $$
Then
\begin{align*}
\widehat{AK(t)}(s)\varphi_n &= -\left(\int_{-\infty}^{\infty} e^{-\mathrm{i}st}\,\lambda_n k_n(t)\,\d t\right)\varphi_n\\
&= -\left(\int_0^\infty e^{-\mathrm{i}st}\,\lambda_n t^{\alpha-1}E_{\alpha,\alpha}(-\mathrm{i}\lambda_n t^\alpha)\,\d t\right)\varphi_n\\
&=-\frac{\lambda_n}{(\mathrm{i}s)^\alpha +\mathrm{i}\lambda_n}\varphi_n\\
& = \frac{\mathrm{i}\lambda_n}{-\mathrm{i}(\mathrm{i}s)^\alpha + \lambda_n}\,\varphi_n, \qquad s\in \mathbb{R}^*.
\end{align*}
Here, $s^\alpha$ for $s<0$ is defined by taking the principal branch. Thus
$$\widehat{AK(t)}(s)= -\mathrm{i}AR(-\mathrm{i}(\mathrm{i}s)^\alpha,A)=-(\mathrm{i}s)^\alpha R(-\mathrm{i}(\mathrm{i}s)^\alpha,A)+\mathrm{i}I_H=:m(s), \; s\in \mathbb{R}^*.$$
Then
$$s\,m'(s)=-\alpha (\mathrm{i}s)^\alpha R(-\mathrm{i}(\mathrm{i}s)^\alpha,A)-\alpha \mathrm{i} [(\mathrm{i}s)^\alpha R(-\mathrm{i}(\mathrm{i}s)^\alpha,A)]^2.$$
Since $A$ is self-adjoint and negative, it generates a bounded analytic semigroup on $H$. Therefore, for $\theta\in (\frac{\pi}{2},\pi)$ there is $M>0$ such that
$$\Sigma_{\theta}:=\{z \in \mathbb{C}^* \colon |\arg (z)|<\theta\} \subset \rho(A)$$
and
\begin{equation}\label{rest}
\|R(z, A)\|_{\mathcal{B}(H)} \leq \frac{M}{|z|}, \qquad z \in \Sigma_{\theta}.
\end{equation}
For $z=-\mathrm{i}(\mathrm{i}s)^\alpha$, we have
$$\arg(z)=\begin{cases}
-\frac{\pi}{2}(1-\alpha) \qquad\mbox{ if } s>0,\\
-\frac{\pi}{2}(1+\alpha) \qquad\mbox{ if } s<0.
\end{cases}$$
For all $0<\alpha<1,$ we take $\theta=\frac{\pi}{2}(\frac{1+\alpha}{2}+1)\in (\frac{\pi}{2},\pi)$ so that $z\in \Sigma_{\theta}$ for all $s\neq 0$.
Therefore, the resolvent estimate \eqref{rest} implies that $m(s)$ and $s\,m'(s)$ are bounded on $\mathbb{R}^*$ with values in $\mathcal{B}(H)$. Then, Theorem \ref{mikh} implies that the operator $g\mapsto Au(g)$ extends to a bounded operator on $L^p(0,T;H)$ denoted by the same notation. Hence, there is a constant $C>0$ such that
\[
 \|Au\|_{L^p(0,T;H)}\le C\,\|f\|_{L^p(0,T;H)},
\]
which yields the maximal $L^p$-regularity.
\end{proof}

\begin{rmk}\label{rmkcs}
Theorem \ref{thmMRL^p} does not hold for the integer derivative case, that is, when $\alpha=1$. Indeed, in this case, maximal regularity would imply that the operator $\mathrm{i} A$ generates an analytic semigroup; see \cite{Dore0}. This is impossible, since the spectrum $\sigma(\mathrm{i}A)=\{-\mathrm{i}\lambda_n\} \subset \mathrm{i}\mathbb{R}$ is unbounded and therefore cannot be contained in any sector of angle $<\frac{\pi}{2}$ on the left half-plane.
\end{rmk}

Next, we discuss the maximal $L^p$-regularity of the homogeneous problem
\begin{equation}\label{heqp}
\partial_t^\alpha (u-u_0) -\mathrm{i} Au=0, \qquad t \in (0,T).
\end{equation}

Henceforth, for $\beta >\frac{1}{p},$ we set
\begin{equation}\label{Interp}
\mathcal{X}_{\beta,p} :=(H,D(A))_{1-\frac{1}{\beta p},\,p}.
\end{equation}

We will use the Fourier method and the weak Lebesgue spaces $L^{p,\infty}(0,T;V)$. Following \cite{Ji21}, we prove:
\begin{theorem}\label{thmMRh}
The solution to \eqref{heqp}, given by
\begin{equation}\label{Formula SEH}
u(t)=\sum_{n=1}^{\infty}\left\langle u_0, \varphi_n\right\rangle E_{\alpha, 1}\left(-\mathrm{i} \lambda_n t^\alpha\right) \varphi_n,
\end{equation}
satisfies the following maximal $L^p$-regularity estimates:
\begin{align}
&\|\partial_t^\alpha (u-u_0)\|_{L^p(0,T;H)}
	+ \|Au\|_{L^p(0,T;H)}
	\le C \|u_0\|_{\mathcal{X}_{\alpha,p}},
    \qquad p\in\left(\tfrac1\alpha,\infty\right], \label{HML1}\\
%&\|\partial_t^\alpha (u-u_0)\|_{L^{\frac{1}{\alpha},\infty}(0,T;H)}
%	+ \|Au\|_{L^{\frac{1}{\alpha},\infty}(0,T;H)}
%	\le C \|u_0\|,
%    \quad p=\tfrac1\alpha, \label{HML2}\\
&\|\partial_t^\alpha (u-u_0)\|_{L^p(0,T;H)}
	+ \|Au\|_{L^p(0,T;H)}
	\le C \|u_0\|,
	 \qquad\qquad p\in\left[1,\tfrac1\alpha\right). \label{HML3}
\end{align}
Here the constant $C>0$ depends on $\alpha, p$ and $T$.
\end{theorem}
\begin{rmk}
Theorem \ref{thmMRh} improves the regularity result in \cite[Theorem 3.3]{CEMY25}, which yields the maximal $L^2$-regularity for $u_0$ belonging to the smaller space $D\left((-A)^\frac{1}{2}\right)$. Indeed, for $p=2$ and $\alpha>\frac{1}{2}$, we have
$$\mathcal{X}_{\alpha,2}\supset D\left((-A)^{1-\frac{1}{2\alpha}}\right) \supset D\left((-A)^\frac{1}{2}\right),$$
since $\alpha<1$.
\end{rmk}
\begin{proof}[Proof of Theorem \ref{thmMRh}]
Using the formula \eqref{Formula SEH} and the estimate \eqref{es0}, we obtain
\begin{align*}
	&\|u(t)\|\leq C_0 \|u_0\| \\
	&\|Au(t)\|\leq C_0 \|Au_0\|, \qquad u_0\in D(A),
\end{align*}
and
\begin{align*}
	\|Au(t)\|^2 &= \sum_{n=1}^\infty \lambda_n^2 |E_{\alpha, 1}\left(-\mathrm{i} \lambda_n t^\alpha\right)|^2 |\langle u_0,\varphi_n\rangle|^2\\
	&\le C_0^2 t^{-2\alpha} \sum_{n=1}^\infty\frac{\lambda_n^2 t^{2\alpha}}{1+\lambda_n^2 t^{2\alpha}}|\langle u_0,\varphi_n\rangle|^2\\
	&\le C_0^2 t^{-2\alpha} \sum_{n=1}^\infty|\langle u_0,\varphi_n\rangle|^2\\
	& =C_0^2 t^{-2\alpha}\|u_0\|^2.
\end{align*}
Then
\begin{equation}\label{MRH3}
	\|Au(t)\|\le C_0 t^{-\alpha}\|u_0\|.
\end{equation}
This implies the estimate \eqref{HML3}.

Now, we prove $\|Au\|_{L^{\frac{1}{\alpha},\infty}(0,T;H)}
\le C_0 \|u_0\|$. For $\lambda>0$ and $t\in (0,T)$, using \eqref{MRH3}, we have $\|Au(t)\|>\lambda$ implies $0<t< C_0^{\frac{1}{\alpha}} \|u_0\|^{\frac{1}{\alpha}}\lambda^{^{-\frac{1}{\alpha}}}.$ Then
\begin{align*}
	\|Au\|_{L^{\frac{1}{\alpha},\infty}(0,T;H)}&=\sup_{\lambda>0} \, \lambda\Big|\Big\{ t\in (0,T)\colon \|Au(t)\|>\lambda \Big\}\Big|^{\alpha}\\
	&\leq C_0\|u_0\|. 
\end{align*}
%This shows the estimate \eqref{HML2}.

Next, for $p>\frac{1}{\alpha}$. Using the two estimates:
\begin{align*}
	\|u\|_{L^{\infty}(0,T;D(A))}&\leq C_0\|u_0\|_{D(A)},\\
	\|u\|_{L^{\frac{1}{\alpha},\infty}(0,T;D(A))}&\leq C \|u_0\|,
\end{align*}
we obtain \eqref{HML1} for $p=\infty$. For $p\in (\frac{1}{\alpha},\infty)$, the real interpolation of the last two estimates gives
\[
\|u\|_{\left(L^{\frac{1}{\alpha},\infty}(0,T;D(A)),\,L^\infty(0,T;D(A))\right)_{1-\frac{1}{\alpha p},\,p}}
\le c \|u_0\|_{\mathcal{X}_{\alpha,p}}.
\]
By \cite[Theorem~2, \S 1.18.6]{Tr78}, we have
\[
\left(L^{\frac{1}{\alpha},\infty}(0,T;D(A)),\,L^\infty(0,T;D(A))\right)_{1-\frac{1}{\alpha p},\,p}
= L^p(0,T;D(A)).
\]
This proves \eqref{HML1} for $p\in\left(\tfrac1\alpha,\infty\right)$.
\end{proof}

By Duhamel's principle, we obtain the maximal $L^p$-regularity for the inhomogeneous fractional Schrödinger equation
\begin{equation}\label{pifseq}
\partial_t^\alpha (u-u_0) -\mathrm{i} Au=f, \qquad t \in (0,T).
\end{equation}
\begin{defn}
Let $1<p<\infty$. We say that the equation $\eqref{pifseq}$ satisfies the maximal $L^p$-regularity, if for all $f\in L^p(0,T;H)$ the corresponding solution satisfies $u-u_0\in W_{\alpha,p}(0,T;H)$. That is, the weak solution $u$ to $\eqref{pifseq}$ belongs to $L^p(0,T;D(A))$, where $D(A)$ is endowed with the graph norm.
\end{defn}
In this case, by the Closed Graph Theorem, there is $C>0$ such that
\begin{equation}
\|u-u_0\|_{W_{\alpha,p}(0,T;H)}+\|Au\|_{L^p(0,T;H)} \le C \|f\|_{L^p(0,T;H)}.
\end{equation}
\begin{theorem}\label{pithmMR}
Let $1<p<\infty, \; \alpha >\frac{1}{p}$, $u_0\in \mathcal{X}_{\alpha,p}$ and $f\in L^p(0,T;H).$ Then \eqref{pifseq} satisfies the maximal $L^p$-regularity. That is, the solution given by
\begin{equation*}
u(t)=\sum_{n=1}^{\infty}\left[\left\langle u_0, \varphi_n\right\rangle E_{\alpha, 1}\left(-\mathrm{i} \lambda_n t^\alpha\right) + \int_0^t \langle f(s),\varphi_n\rangle\,(t-s)^{\alpha-1}E_{\alpha,\alpha}\big(-\mathrm{i}\lambda_n (t-s)^\alpha\big)\,\d s \right]\varphi_n,
\end{equation*}
satisfies $u\in L^p(0,T;D(A))$, and there exists a constant $C>0$ such that
\[
\|u-u_0\|_{W_{\alpha,p}(0,T;H)} + \|u\|_{L^p(0,T;D(A))}\le C\left(\|u_0\|_{\mathcal{X}_{\alpha,p}}+\|f\|_{L^p(0,T;H)}\right).
\]
\end{theorem}
\begin{rmk}
The assumption $\alpha>\frac{1}{p}$ implies that the solution of \eqref{pifseq} belongs to $C([0,T];H)$ (see Proposition \ref{propc}) so that the initial datum $u(0)=u_0$ makes sense.
\end{rmk}

\section{Application to nonlinear Schr\"odinger equations}\label{sec6}
In this section, we apply the maximal regularity results of the previous section to study the well-posedness of some quasilinear and semilinear Schr\"odinger equations. We refer, for example, to \cite{Mey14} for quasilinear reaction-diffusion systems with an integer derivative.

To simplify the notation, for fixed $\alpha\in (0,1)$ and $p>\frac{1}{\alpha},$ we set
\begin{equation}\label{MRsp}
    \mathcal{M}\mathcal{R}_{\alpha,p}(T):=\left\{u\in C([0,T];H)\cap L^p(0,T;D(A))\colon  u-u(0)\in W_{\alpha,p}(0,T;H) \right\}
\end{equation}
for the maximal regularity space. By Proposition \ref{fractional_space}, the space $\mathcal{M}\mathcal{R}_{\alpha,p}(T),$ endowed with the norm
\[
\|u\|_{\mathcal{M}\mathcal{R}_{\alpha,p}(T)} := \|u-u(0)\|_{W_{\alpha,p}(0,T;H)} + \|u\|_{L^p(0,T;D(A))}+ \|u(0)\|, \qquad u\in \mathcal{M}\mathcal{R}_{\alpha,p}(T) ,
\]
is a Banach space.

\subsection{Quasilinear equations}
We consider the quasilinear fractional Schr\"odinger equation
\begin{equation}\label{nfseq0}
\partial_t^\alpha (u(t)-u_0) -\mathrm{i} \mathcal{A}(u(t))u(t)=0, \qquad t \in (0,T),
\end{equation}
where $\alpha\in (0,1)$ and $p>\frac{1}{\alpha}$.

Next, we choose $0<\varepsilon<\alpha-\frac{1}{p}$ and introduce the following assumption for the quasilinear term:
\begin{itemize}
    \item[\textbf{(H$_\varepsilon$)}] $\mathcal{A}(0)=A$ and $\mathcal{A}\colon \mathcal{X}_{\alpha-\varepsilon,p} \to \mathcal{B}(D(A),H)$ is Lipschitz continuous on bounded subsets of $\mathcal{X}_{\alpha-\varepsilon,p}$.
\end{itemize}
Then, we prove the following key lemma.
\begin{lem}
The continuous embedding
\[
\mathcal{M}\mathcal{R}_{\alpha,p}(T)
\hookrightarrow C\left([0,T];\mathcal{X}_{\alpha-\varepsilon,p}\right)
\]
holds, i.e., there is a constant \(C_\mathrm{em}>0\), depending on \(\alpha,p,\varepsilon,T\), such that for every
\(u\in \mathcal{M}\mathcal{R}_{\alpha,p}(T)\), we have
\begin{equation}\label{emb}
\|u\|_{C([0,T];\mathcal{X}_{\alpha-\varepsilon,p})}
\le C_\mathrm{em} \|u\|_{\mathcal{M}\mathcal{R}_{\alpha,p}(T)}.
\end{equation}
\end{lem}
\begin{proof}
First, following \cite[Theorem 2.1~(i)]{YaLp22}, we can prove that the embedding
\begin{equation}\label{SSlob}
W_{\alpha,p}(0,T;H) \hookrightarrow W^{\alpha-\varepsilon,p}(0,T;H)
\end{equation}
is continuous, where $W^{\alpha-\varepsilon,p}(0,T;H)$ is the Sobolev-Slobodeckii space (see Adams \cite{Ad75}). Then, by $(\alpha-\varepsilon)p>1$ and extension by zero to $(0,\infty),$ \cite[Theorem 1.2]{Ve23} yields
$$W^{\alpha-\varepsilon,p}(0,T;H)\cap L^p(0,T;D(A))
\hookrightarrow
C\left([0,T];\mathcal{X}_{\alpha-\varepsilon,p}\right)$$
continuously. Let $u\in \mathcal{M}\mathcal{R}_{\alpha,p}(T)$ and $v:=u-u(0)\in W_{\alpha,p}(0,T;H)$. We have $v\in W^{\alpha-\varepsilon,p}(0,T;H)$ and $u(0)\in W^{\alpha-\varepsilon,p}(0,T;H).$ Then $u\in W^{\alpha-\varepsilon,p}(0,T;H)\cap L^p(0,T;D(A)).$ Therefore, $u\in C\left([0,T];\mathcal{X}_{\alpha-\varepsilon,p}\right)$ and
\begin{align*}
\|u\|_{C\left([0,T];\mathcal{X}_{\alpha-\varepsilon,p}\right)} & \le C \left(\|u\|_{W^{\alpha-\varepsilon,p}(0,T;H)}+\|u\|_{L^p(0,T;D(A))}\right)\\
& \le C \left(\|u-u_0\|_{W^{\alpha-\varepsilon,p}(0,T;H)}+T^{\frac{1}{p}}\|u_0\|+\|u\|_{L^p(0,T;D(A))}\right)\\
&\le C \|u\|_{\mathcal{M}\mathcal{R}_{\alpha,p}(T)}.
\end{align*}
This yields the desired result.
\end{proof}

Now, we state the well-posedness of \eqref{nfseq0} for arbitrary $T$ and small initial data.
\begin{theorem}\label{thm:quasi}
Let $T>0$ be fixed, and assume the assumption \textbf{(H$_\varepsilon$)}. Then there exists $\varepsilon^\prime>0$ such that for all $u_0\in \overline{B}_{\mathcal{X}_{\alpha,p}}(\varepsilon^\prime)$,
%$u_0\in \mathcal{X}_{\alpha,p}$ such that $ \|u_0\|_{\mathcal{X}_{\alpha,p}} \le \varepsilon^\prime,$ 
the equation \eqref{nfseq0} has a unique solution $u\in \mathcal{M}\mathcal{R}_{\alpha,p}(T).$
\end{theorem}
\begin{proof}
We define the map $\mathcal{L}(f,u_0)=u,$ where $u$ is the solution to \eqref{fseq0}. By the maximal regularity (Theorem \ref{pithmMR}), we have $$\mathcal{L}\in \mathcal{B}\left(L^p(0,T;H)\times \mathcal{X}_{\alpha,p}, \mathcal{M}\mathcal{R}_{\alpha,p}(T)\right).$$ 
We define
$$\Psi(u)(t):=\mathrm{i}(\mathcal{A}(u(t))-A)u(t) \qquad \text{ and } \qquad \Phi(u):=\mathcal{L}(\Psi(u), u_0).$$
Then $u$ solves \eqref{nfseq0} if and only if $\Phi(u)=u$. Therefore, we will apply the Banach contraction mapping theorem. We set
%$$B_r:=\left\{v\in \mathcal{M}\mathcal{R}_{\alpha,p}(T) \colon \|v\|_{\mathcal{M}\mathcal{R}_{\alpha,p}(T)} \le r\right\},$$
$$B_r:=\overline{B}_{\mathcal{M}\mathcal{R}_{\alpha,p}(T)}(r),$$
where $r>0$ is sufficiently small to be fixed later. First, we prove that $\Phi (B_r)\subset B_r$. For $v\in B_r,$ by \eqref{emb} and $\mathcal{A}(0)=A$, we have 
\begin{align*}
\|\mathcal{A}(v(t))-A\|_{\mathcal{B}(D(A),H)} &\le L_\mathcal{A} \|v(t)\|_{\mathcal{X}_{\alpha-\varepsilon,p}}\\
&\le L_\mathcal{A} C_{\mathrm{em}} \|v\|_{\mathcal{M}\mathcal{R}_{\alpha,p}(T)}\\
&\le  L_\mathcal{A} C_{\mathrm{em}} r\\
& \le \frac{1}{4\|\mathcal{L}\|}
\end{align*}
for small $r$. Then
\begin{align*}
&\|\Phi(v)\|_{\mathcal{M}\mathcal{R}_{\alpha,p}(T)} \le \|\mathcal{L}\| \left(\|\Psi(v)\|_{L^p(0,T;H)}+\|u_0\|_{\mathcal{X}_{\alpha,p}}\right)\\
&\le \|\mathcal{L}\| \left(\max_{t\in [0,T]} \|\mathcal{A}(v(t))-A\|_{\mathcal{B}(D(A),H)}\|v\|_{\mathcal{M}\mathcal{R}_{\alpha,p}(T)}+\|u_0\|_{\mathcal{X}_{\alpha,p}}\right)\\
& \le \|\mathcal{L}\| \left(\frac{1}{4\|\mathcal{L}\|} r +\frac{1}{4\|\mathcal{L}\|} r\right)=\frac{r}{2}<r,
\end{align*}
if we choose $\varepsilon^\prime:=\dfrac{r}{4\|\mathcal{L}\|}$. Now, we prove that $\Phi$ is a contraction on $B_r$ in a similar way. For $v_1, v_2\in B_r,$ we have
\begin{align*}
&\|\Phi(v_1)-\Phi(v_2)\|_{\mathcal{M}\mathcal{R}_{\alpha,p}(T)} \le \|\mathcal{L}\| \|\Psi(v_1)-\Psi(v_2)\|_{L^p(0,T;H)}\\
&\le \|\mathcal{L}\| \left(\max_{t\in [0,T]} \|\mathcal{A}(v_1(t))-A\|_{\mathcal{B}(D(A),H)} \|v_1-v_2\|_{\mathcal{M}\mathcal{R}_{\alpha,p}(T)}\right.\\
& \hspace{1.3cm} \left. + \max_{t\in [0,T]} \|\mathcal{A}(v_1(t))-\mathcal{A}(v_2(t))\|_{\mathcal{B}(D(A),H)} \|v_2\|_{\mathcal{M}\mathcal{R}_{\alpha,p}(T)}\right)\\
& \le \|\mathcal{L}\| \left(\frac{1}{4\|\mathcal{L}\|} \|v_1-v_2\|_{\mathcal{M}\mathcal{R}_{\alpha,p}(T)} + L_\mathcal{A} C_{\mathrm{em}} r \|v_1-v_2\|_{\mathcal{M}\mathcal{R}_{\alpha,p}(T)}\right)\\
& \le\frac{1}{2} \|v_1-v_2\|_{\mathcal{M}\mathcal{R}_{\alpha,p}(T)}.
\end{align*}
By Banach's contraction mapping theorem, there is a unique solution $u\in \mathcal{M}\mathcal{R}_{\alpha,p}(T).$
\end{proof}

\begin{rmk}
In the case $p=2$ and $\alpha>\frac{1}{2},$ we can choose $\varepsilon=0$ in the assumption \textbf{(H$_\varepsilon$)}. Indeed, it has been proven in \cite[Theorem 2.4.1]{Ya25} that the continuous embedding
$$H_\alpha(0,T;H) \hookrightarrow H^\alpha(0,T;H)$$
holds, which is stronger than the embedding \eqref{SSlob}. However, in the case $p\neq 2,$ we do not
know if we have a continuous embedding
$$W_{\alpha,p}(0,T;H) \hookrightarrow W^{\alpha,p}(0,T;H).$$
\end{rmk}
\smallskip

We now present an example for which Theorem \ref{thm:quasi} applies in a semilinear setting.

\subsubsection*{\textbf{Example.}}
We consider \(H = L^2(\Omega)\) as a base space, where $\Omega \subset \mathbb{R}^d\;(d \le 3)$ is a bounded domain with smooth boundary. We consider the following time-fractional Schrödinger equation with nonlinear diffusivity and reaction terms
\begin{equation}\label{eq_ex}
\begin{cases}
\mathrm{i}~\partial_t^\alpha \big(u-u_0\big) +\nabla\cdot\big(a(u) \nabla u\big) = f(u), 
& \text{in } (0,T)\times \Omega,\\
u = 0, & \text{in } (0,T)\times \partial \Omega,
\end{cases}
\end{equation}
where $\nabla\cdot$ denotes the divergence operator. Here, we assume that:
\begin{itemize}
    \item the diffusivity coefficient satisfies
		\begin{equation}\label{assumption 1}
			a \in W^{2,\infty}_{\mbox{loc}}(\mathbb{C},\mathbb{R})\qquad \mbox{and}\qquad a(0)>0;
		\end{equation}
		\item the reaction term satisfies
		\begin{equation}\label{assumption 2}
			f\in W^{2,\infty}_{\mbox{loc}}(\mathbb{C},\mathbb{C})\qquad \mbox{and}\qquad f(0)=0.
		\end{equation}
    \end{itemize}
A classical example of a reaction term is the cubic nonlinearity, which corresponds to a bistable reaction: \( f(u) = u - |u|^2 u. \)
        
In this case, $\mathcal{A}(u)(v)=\nabla\cdot(a(u)\nabla v)-g(u)v$, where $$g(\xi)=\begin{cases}
	\frac{f(\xi)}{\xi}, \quad &\mbox{if}\quad \xi\neq 0,\\
	f'(0), \quad &\mbox{if}\quad \xi=0,
\end{cases}$$
and $A=\mathcal{A}(0)=a(0)\Delta-g(0),$ with domain $D(A)=H^2(\Omega)\cap H^1_0(\Omega)$. 
We further assume $g(0)=f'(0)\ge 0$ so that $\sigma(-A)\subset (0,\infty)$. Note that $g \in W^{1,\infty}_{\mbox{loc}}(\mathbb{C},\mathbb{C})$ with Lipschitz constant $L_g =\frac{L_{f^\prime}}{2}$ due to $g(\xi)=\displaystyle\int_0^1 f^\prime(s\xi)\d s,\; \xi\in\mathbb{C}$.

We can also consider the equation in non-divergence form:
\begin{equation*} 
\begin{cases} \mathrm{i}~\partial_t^\alpha \big(u-u_0\big) + a(u)\Delta u = f(u), & \text{in } (0,T)\times \Omega,\\ u = 0, & \text{in } (0,T)\times \partial \Omega, \end{cases} 
\end{equation*}
for which it is sufficient to assume that $a\in W^{1,\infty}_{\mbox{loc}}(\mathbb{C},\mathbb{R})$.

\begin{prop}
Let $T>0$, $\alpha\in (0,1)$ be fixed, and assume \eqref{assumption 1}-\eqref{assumption 2}. Then there exist $p>\frac{1}{\alpha}$ and $\varepsilon^\prime>0$ such that for every $u_0\in \overline{B}_{\mathcal{X}_{\alpha,p}}(\varepsilon^\prime)$,
%$u_0\in \mathcal{X}_{\alpha,p}$ with $ \|u_0\|_{\mathcal{X}_{\alpha,p}} \le \varepsilon',$ 
the equation \eqref{eq_ex} has a unique solution $u\in \mathcal{M}\mathcal{R}_{\alpha,p}(T).$
\end{prop}
\begin{proof}
We will prove that the quasilinear operator $\mathcal{A}$ satisfies the assumption $\textbf{(H$_\varepsilon$)}.$ For this purpose, we first prove the continuous embedding 
\begin{equation}\label{embounded}
    \mathcal{X}_{\alpha-\varepsilon,p}\hookrightarrow L^{\infty}(\Omega)\cap W^{1,4}(\Omega)
\end{equation}
for $p$ large enough, to estimate nonlinear gradient terms. Indeed, using the embedding $D(A)\hookrightarrow H^2(\Omega)$, we obtain 
\begin{equation}\label{embesov}
    \mathcal{X}_{\alpha-\varepsilon,p}\hookrightarrow (L^2(\Omega), H^2(\Omega))_{1-\frac{1}{(\alpha-\varepsilon) p},\,p}=B_{2,p}^{s}(\Omega),
\end{equation}
where $B_{2,p}^{s}(\Omega)$ denotes the Besov space (see, e.g., Adams and Fournier \cite[\S 7.30]{AF03}), and $s=2\left(1-\frac{1}{(\alpha-\varepsilon) p}\right)$.

Now, we find $p$ such that $B_{2,p}^{s}(\Omega)\hookrightarrow W^{1,4}(\Omega)$. Let $u\in B_{2,p}^{s}(\Omega)$ and $\epsilon\in [\frac{1}{2},1)$. By $B_{2,p}^{s}(\Omega)\hookrightarrow H^{s-\epsilon}(\Omega)$, we have $u,\;\nabla u\in H^{s-\epsilon-1}(\Omega)$.
Given $d\leq 3$ and \cite[Theorem 6.7]{DiN12}, we obtain that $u\in W^{1,4}(\Omega)$ and $\|u\|_{W^{1,4}(\Omega)}\le C\|u\|_{H^{s-\epsilon-1}(\Omega)},$ provided that $\frac{d}{d-2(s-\epsilon-1)} \ge 2.$ This condition is satisfied by choosing $p\ge \frac{8d}{(4(1-\epsilon)-d)(\alpha-\varepsilon)}=:p_1.$ Moreover, choosing $p>\frac{4}{(4-d)(\alpha-\varepsilon)}=:p_2,$ we have by \cite[Theorem 4.6.1, p.327]{Tr78}, $B_{2,p}^{s}(\Omega)\hookrightarrow L^\infty(\Omega).$ Therefore, by \eqref{embesov}, for all $p>\max(\frac{1}{\alpha-\varepsilon},p_1,p_2),$ we obtain the claim \eqref{embounded}.

Now, let $B\subset \mathcal{X}_{\alpha-\varepsilon,p}$ be a bounded set, and show that $\mathcal{A}$ is Lipschitz continuous on $B$. By the Sobolev embedding \eqref{embounded}, there is $r_0>0$ such that $B\subset \overline{B}_{L^{\infty}(\Omega)\cap W^{1,4}(\Omega)}(r_0)$.
%$\|z\|_{L^{\infty}(\Omega)\cap W^{1,4}(\Omega)}\leq c_{0}$ for all $z\in B$. 
Let $y\in D(A)$ and $z, w\in B$. Obviously, we have
\begin{align*}
	\mathcal{A}(z)y-\mathcal{A}(w)y=&(a(z)-a(w))\Delta y + (a^{\prime}(z)-a^{\prime}(w))\nabla z\cdot\nabla y \\
    &+a^{\prime}(w)\nabla(z-w)\cdot\nabla y-(g(z)-g(w))y.
\end{align*}
Using $a$, $a^{\prime}$ and $f^\prime$ are Lipschitz on the closed disk $\overline{D}(0,r_0)$ with Lipschitz constants denoted by $L_a$, $L_{a'}$, and $L_{f'}$, Hölder's inequality and the above Sobolev embedding $H^{2}(\Omega)\hookrightarrow W^{1,4}(\Omega)$ (see, e.g., \cite[Theorem 4.12]{AF03}), we obtain
			\begin{align*}
				&\|\mathcal{A}(z)y-\mathcal{A}(w)y\|_{L^{2}(\Omega)}\\
				&\leq L_a\|z-w\|_{L^{\infty}(\Omega)}\|y\|_{H^2(\Omega)} +L_{a^\prime} \|z-w\|_{L^{\infty}(\Omega)}\| z\|_{W^{1,4}(\Omega)}\|y\|_{W^{1,4}(\Omega)} \\
				& \quad +\|a^\prime\|_{L^\infty(\overline{D}(0,r_0))}\|z-w\|_{W^{1,4}(\Omega)}\| y\|_{W^{1,4}(\Omega)} + \frac{L_{f^\prime}}{2}\|z-w\|_{L^{\infty}(\Omega)}\|y\|_{L^{2}(\Omega)}\\
				%&\leq C \left(\|z-w\|^{2}_{\mathcal{X}_{\alpha-\varepsilon,p}}\|y\|^{2}_{H^{2}(\Omega)} + c_0^2\|z-w\|^{2}_{\mathcal{X}_{\alpha-\varepsilon,p}}\|y\|^{2}_{H^{2}(\Omega)}  \right.\\
				%&\left. \quad +\| z-w\|^{2}_{\mathcal{X}_{\alpha-\varepsilon,p}}\| y\|^{2}_{H^{2}(\Omega)} +\| z-w\|^{2}_{\mathcal{X}_{\alpha-\varepsilon,p}}\| y\|^{2}_{H^{2}(\Omega)}\right)\\
				&\leq C\|z-w\|_{\mathcal{X}_{\alpha-\varepsilon,p}}\|y\|_{H^{2}(\Omega)},
			\end{align*}
            where $C>0$ depends on $a$, $f$, and $B$.
\end{proof}

%\noindent\textbf{Example.} Let $\Omega \subset \mathbb{R}^d$ be a bounded domain with a smooth boundary. We consider $H=L^2(\Omega)$ with the standard inner product. We set $A=\Delta_D$ for the Dirichlet Laplacian with domain $D(A)=H^2(\Omega)\cap H^1_0(\Omega)$ and denote
%$$\mathcal{A}(u(t))=\Delta_D-**************.$$
%\textcolor{red}{
%$$\mathcal{A}(u)(v)=\nabla\cdot(a(u)\nabla v).$$
%where ....................
%}\\
%In this case, for $p=2$ and $\alpha>\frac{1}{2},$ the assumption %\textbf{(H$_\varepsilon$)} is satisfied with $\varepsilon=0$.

\subsection{Semilinear equations}
We consider the following semilinear fractional Schr\"odinger equation
\begin{equation}\label{semieq0}
\partial_t^\alpha (u(t)-u_0) -\mathrm{i} Au(t)=F(u(t)), \qquad t \in (0,T),
\end{equation}
where we consider the following assumption:
\begin{itemize}
     \item[\textbf{(H)}] $F: \mathcal{M}\mathcal{R}_{\alpha,p}(T) \to L^p(0,T;H)$ satisfies the following continuity property: for any $\epsilon>0$, there
exists a constant $C=C(\epsilon,p)>0$ such that
     \begin{equation}
        \|F(v)-F(w)\|_{L^p(0,T;H)}\le \epsilon\|v-w\|_{\mathcal{M}\mathcal{R}_{\alpha,p}(T)}+C\|v-w\|_{L^p(0,T;H)}
    \end{equation}
     %\begin{equation}
      %  \|F(v)-F(w)\|_{L^p(0,T;H)}\le \epsilon\|v-w\|_{\mathcal{M}\mathcal{R}_{\alpha,p}(T)}+C\|(v-w)-(v(0)-w(0))\|_{L^p(0,T;H)}
    %\end{equation}
    for all $v,w\in \mathcal{M}\mathcal{R}_{\alpha,p}(T).$
\end{itemize}
A sufficient condition under which this hypothesis holds is given by the following (see \cite[Lemme 8.2]{Are15}):
    \begin{equation}
        F\in \mathcal{B}(\mathcal{M}\mathcal{R}_{\alpha,p}(T), L^p(0,T;H))\; \mbox{is compact}.
    \end{equation}
    Another simple assumption implying \textbf{(H)} is the following (see \cite[Example 4.1]{Ach22}):
   $$F : D(A) \to H \quad \text{is globally Lipschitz}.$$
\begin{rmk}
Although we follow \cite{Ach22} in this part, we address the following issues:
\begin{itemize}
    \item In \cite[p.~14]{Ach22}, the maximal regularity space $MR(\alpha,p,\tau)$ depends on a fixed initial datum $u_0$ (see also \cite[p.~8]{Za09} and \cite[p.~528]{Ach24}).
    \item The trace space $Tr^p_\alpha$ of initial data is not well defined (see \cite[p.~14]{Ach22}). The natural space for $\alpha>\frac{1}{p}$ is given by $\mathcal{X}_{\alpha,p}$ defined in \eqref{Interp}.
    \item The associated quantity $\|\cdot\|_{MR(\alpha,p,\tau)}$ in \cite[p.~14]{Ach22} is not a norm when $\alpha<\frac{1}{p}$ and $u_0\neq 0$.
\end{itemize}
Therefore, the results in \cite{Ach22} hold only for $u_0=0$. This is not the case for our maximal regularity space $\mathcal{MR}_{\alpha,p}(T)$ defined in \eqref{MRsp}.
\end{rmk}

Next, we state the well-posedness of \eqref{semieq0} for small times.
\begin{theorem}\label{ThmAch}
Let $\alpha\in (0,1)$, $p>\frac{1}{\alpha}$ and assume the assumption \textbf{(H)}. Then, there exists $T_1>0$ such that, for all $u_0\in \mathcal{X}_{\alpha,p}$ and all $T\in (0,T_1)$, the equation \eqref{semieq0} has a unique solution $u\in \mathcal{M}\mathcal{R}_{\alpha,p}(T).$ %Moreover, there exists a constant $C=C(\alpha,p)>0$, independent of $u_0$ and $F$, such that
%\[
%\|u\|_{\mathcal{M}\mathcal{R}_{\alpha,p}(T)}
%\le
%C
%\left(\|u_0\|+\|F(0)\|_{L^{p}(0,T;H)}
%\right).\]
\end{theorem}
\begin{rmk}
In the above theorem, we only consider small times, in contrast to \cite[Theorem~4.2]{Ach22}, whose proof omits the initial datum in the final estimate on p.~20.
\end{rmk}

To prove Theorem \ref{ThmAch}, we need some preliminary estimates. For $\alpha>0$, we set
\[
k_\alpha(t)=\begin{cases}
    t^{\alpha-1},\quad & t>0,\\
    0, \quad & t\le 0,
\end{cases}
\]
and we denote by $k_\alpha^{*n}$ the $n$-fold convolution of $k_\alpha$ with itself, that is,
\[
k_\alpha^{*n}
:= \underbrace{k_\alpha * k_\alpha * \cdots * k_\alpha}_{n\ \text{times}}.
\]
Then, for all $t>0$, the following explicit formula holds
\begin{equation}
    k_\alpha^{*n}(t)
= \frac{\Gamma(\alpha)^n}{\Gamma(n\alpha)}\, t^{n\alpha-1}. \label{n_convolution}
\end{equation}
We prove the following key lemma.
\begin{lem}
    For all $u\in W_{\alpha,p}(0,T;H)$, we have
    \begin{align}
		\int_0^T \|u(t)\|^p\d t
		&\le \frac{T^{\frac{\alpha p}{q}}}{\alpha^{\frac{p}{q}}\Gamma(\alpha)^p}\int^T_0 k_{\alpha}(T-r)
		\int_0^r\|\partial_t^\alpha u(s)\|^p\;\d s\;\d r. \label{lm:semi}
	\end{align}
\end{lem}
\begin{proof}
    Let $u\in W_{\alpha,p}(0,T;H)$. From \eqref{deffd}, we have
	\begin{equation*}
	u=J^\alpha \partial_t^\alpha u,\qquad \mbox{ in} \quad L^p(0,T;H).
	\end{equation*}
	Then, for a.e. $t\in (0,T),$ and by H\"older's inequality, we obtain
	\begin{align*}
		\|u(t)\|&= \|J^\alpha \partial_t^\alpha u (t)\|\\
		%&\le \frac{1}{\Gamma(\alpha)}\int^t_0 k_\alpha(t-s)\|\partial_t^\alpha u(s)\|\, \d s \\
		&\le \frac{1}{\Gamma(\alpha)} \left(\int^t_0 k_\alpha(t-s)
		\, \d s \right)^{\frac{1}{q}} \left(\int^t_0 k_\alpha(t-s)
		\|\partial_t^\alpha u(s)\|^p\, \d s \right)^{\frac{1}{p}}\\
		%&= \frac{t^{\frac{\alpha}{q}}}{\alpha^{\frac{1}{q}}\Gamma(\alpha)}\left(\int^t_0 k_\alpha(t-s)
		%\|\partial_t^\alpha u(s)\|^p\, \d s \right)^{\frac{1}{p}}\\
		&\le \frac{T^{\frac{\alpha}{q}}}{\alpha^{\frac{1}{q}}\Gamma(\alpha)}\left(\int^t_0 k_\alpha(t-s)
		\|\partial_t^\alpha u(s)\|^p\, \d s \right)^{\frac{1}{p}}.
	\end{align*}
	Then integrating over $(0,T)$ and applying Tonelli’s theorem, we obtain 
    \begin{align*}
        \int_0^T \|u(t)\|^p\d t
		&\le \frac{T^{\frac{\alpha p}{q}}}{\alpha^{\frac{p}{q}}\Gamma(\alpha)^p}\int_0^T\int^t_0 k_\alpha(t-s)
		\|\partial_t^\alpha u(s)\|^p\, \d s\;\d t\\
        &= \frac{T^{\frac{\alpha p}{q}}}{\alpha^{\frac{p}{q}}\Gamma(\alpha)^p}\int^T_0\int_s^T k_\alpha(t-s)
		\|\partial_t^\alpha u(s)\|^p\, \d t\;\d s\\
        &=\frac{T^{\frac{\alpha p}{q}}}{\alpha^{\frac{p}{q}+1}\Gamma(\alpha)^p}\int^T_0 k_{\alpha+1}(T-s)
		\|\partial_t^\alpha u(s)\|^p\;\d s \\
        &=\frac{T^{\frac{\alpha p}{q}}}{\alpha^{\frac{p}{q}}\Gamma(\alpha)^p}\int^T_0\int_s^T k_{\alpha}(T-r)\d r
		\|\partial_t^\alpha u(s)\|^p\;\d s\\
        &= \frac{T^{\frac{\alpha p}{q}}}{\alpha^{\frac{p}{q}}\Gamma(\alpha)^p}\int^T_0 k_{\alpha}(T-r)
		\int_0^r\|\partial_t^\alpha u(s)\|^p\;\d s\;\d r.
    \end{align*}
    This yields the desired estimate.
\end{proof}

\begin{proof}[Proof of Theorem \ref{ThmAch}]
We define the linear operator $\mathcal{L}: L^p(0,T;H)\times\mathcal{X}_{\alpha,p}\to \mathcal{M}\mathcal{R}_{\alpha,p}(T)$ by  $\mathcal{L}(f,u_0)=u,$ where $u$ is the solution to \eqref{fseq0} associated with $u_0$ and $f$. By the maximal regularity (Theorem \ref{pithmMR}), we have $$\mathcal{L}\in \mathcal{B}\left(L^p(0,T;H)\times \mathcal{X}_{\alpha,p}, \mathcal{M}\mathcal{R}_{\alpha,p}(T)\right).$$ 
We fix $u_0\in \mathcal{X}_{\alpha,p}$ and define $\Phi:\mathcal{M}\mathcal{R}_{\alpha,p}(T)\to \mathcal{M}\mathcal{R}_{\alpha,p}(T)$ by
$$\Phi(v):=\mathcal{L}(F(v),u_0).$$
Note that, for all $v\in \mathcal{M}\mathcal{R}_{\alpha,p}(T)$, we have $F(v)\in L^p(0,T;H)$, so $\Phi$ is well defined. Then $u$ solves \eqref{semieq0} if and only if $\Phi(u)=u$. Therefore, we will apply the Banach contraction mapping theorem for iterates of $\Phi$.

Let $v,w\in \mathcal{M}\mathcal{R}_{\alpha,p}(T)$. Then,
 \begin{align*}
     &\|\Phi(v)-\Phi(w)\|_{\mathcal{M}\mathcal{R}_{\alpha,p}(T)}= \|\mathcal{L}(F(v)-F(w),0)\|_{\mathcal{M}\mathcal{R}_{\alpha,p}(T)}\\
     &\le \|\mathcal{L}\| \|F(v)-F(w)\|_{L^p(0,T;H)}\\
     &\le \epsilon \|\mathcal{L}\|\|v-w\|_{\mathcal{M}\mathcal{R}_{\alpha,p}(T)}+C\|\mathcal{L}\|\|v-w\|_{L^p(0,T;H)}\\
     &\le \epsilon \|\mathcal{L}\|\|v-w\|_{\mathcal{M}\mathcal{R}_{\alpha,p}(T)}\\
     &\quad+C\|\mathcal{L}\|\left(\|(v-w)-(v(0)-w(0))\|_{L^p(0,T;H)}+ \|v(0)-w(0)\|_{L^p(0,T;H)}\right)\\
     &= \epsilon \|\mathcal{L}\|\|v-w\|_{\mathcal{M}\mathcal{R}_{\alpha,p}(T)}\\
     &\quad+C\|\mathcal{L}\|\left(\|(v-w)-(v(0)-w(0))\|_{L^p(0,T;H)}+ T^{\frac{1}{p}} \|v(0)-w(0)\|\right).
 \end{align*}
Using the estimate \eqref{lm:semi}, we obtain
\begin{align*}
     &\|\Phi(v)-\Phi(w)\|_{\mathcal{M}\mathcal{R}_{\alpha,p}(T)}\le (\epsilon +CT^{\frac{1}{p}})\|\mathcal{L}\|\|v-w\|_{\mathcal{M}\mathcal{R}_{\alpha,p}(T)}\\
     &+C \|\mathcal{L}\|\frac{T^{\frac{\alpha }{q}}}{\alpha^{\frac{1}{q}}\Gamma(\alpha)^p}\left(\int^T_0 k_{\alpha}(T-r)
		\int_0^r\|\partial_t^\alpha ((v-w)-(v(0)-w(0)))(s)\|^p\;\d s\;\d r\right)^{\frac{1}{p}} \\
    &\le (\epsilon +CT^{\frac{1}{p}})\|\mathcal{L} \|\|v-w\|_{\mathcal{M}\mathcal{R}_{\alpha,p}(T)}
     +C \|\mathcal{L}\|\frac{T^{\frac{\alpha }{q}}}{\alpha^{\frac{1}{q}}\Gamma(\alpha)^p}\left(\int^T_0 k_{\alpha}(T-r)
		\|v-w\|^p_{\mathcal{M}\mathcal{R}(r)}\;\d r\right)^{\frac{1}{p}}\\
        & =\frac{1}{2}\|v-w\|_{\mathcal{M}\mathcal{R}_{\alpha,p}(T)}
     + \gamma\left(\int^T_0 k_{\alpha}(T-r)
		\|v-w\|^p_{\mathcal{M}\mathcal{R}(r)}\;\d r\right)^{\frac{1}{p}},
 \end{align*}
 where we choose $\epsilon=\frac{1}{4\|\mathcal{L}\|}$, $0<T\le \frac{1}{(4 C\|\mathcal{L}\|)^p}$ and set $\gamma:=C \|\mathcal{L}\|\frac{T^{\frac{\alpha }{q}}}{\alpha^{\frac{1}{q}}\Gamma(\alpha)^p}$.
 
Thus, by induction, for every $n\ge 1$, we obtain
\begin{align*}
     \|\Phi^n(v)-\Phi^n(w)\|_{\mathcal{M}\mathcal{R}_{\alpha,p}(T)}
&\le \frac{1}{2^n}\|v-w\|_{\mathcal{M}\mathcal{R}_{\alpha,p}(T)}  \\
&\quad + \gamma
\sum_{j=1}^{n}
\frac{1}{2^{n-j}}
\left(
\int_0^T k_\alpha^{(*j)}(T-r)\,
\|u-v\|_{\mathcal{MR}(r)}^p\,dr
\right)^{\frac{1}{p}}\\
&\hspace{-1cm}\le \left[\frac{1}{2^n}+ \gamma
\sum_{j=1}^{n}
\frac{1}{2^{n-j}}
\left(
\int_0^T k_\alpha^{(*j)}(T-r)\,
\,dr
\right)^{\frac{1}{p}}\right]\|v-w\|_{\mathcal{M}\mathcal{R}_{\alpha,p}(T)}.  
 \end{align*}
 Using the formula \eqref{n_convolution}, we obtain
 \begin{align*}
     \|\Phi^n(v)-\Phi^n(w)\|_{\mathcal{M}\mathcal{R}_{\alpha,p}(T)}
&\le \frac{1}{2^{n}}\left[1+ \gamma
\sum_{j=1}^{n}
2^{j}
\left(
\frac{\Gamma(\alpha)^j}{\Gamma(j\alpha+1)}T^{j\alpha}
\right)^{\frac{1}{p}}\right]\|v-w\|_{\mathcal{M}\mathcal{R}_{\alpha,p}(T)}.
 \end{align*}
The convergence of the series $\displaystyle \sum_{j\ge 1}
2^{j}
\left(
\frac{\Gamma(\alpha)^j}{\Gamma(j\alpha+1)}T^{j\alpha}
\right)^{\frac{1}{p}}$ follows from Stirling's formula and the Cauchy root test. Then, for $n$ sufficiently large, 
$\Phi^n$ is a contraction. By Banach's contraction mapping theorem, there is a unique solution $u\in \mathcal{M}\mathcal{R}_{\alpha,p}(T).$
\end{proof}
The following example is a consequence of the compactness of the embedding (see \cite[Theorem 2.1~(iii)]{YaLp22})
$$
W_{\alpha,p}(0,T;H) \hookrightarrow W_{\beta,p}(0,T;H), 
\qquad 0 < \beta < \alpha.$$
\textbf{Example.} Let $\alpha, \beta\in (0,1)$, $\beta<\alpha$ and $p>\frac{1}{\alpha}$. The following equation is locally well-posed in the sense of Theorem~\ref{ThmAch}:
\begin{equation*}
\begin{cases}
\partial_t^\alpha \big(u-u_0\big) -\mathrm{i} \Delta u
+ \xi \partial_t^\beta (u-u_0)=f(u), 
& \text{in } (0,T)\times \Omega,\\
u = 0, & \text{in } (0,T)\times \partial \Omega,
\end{cases}
\end{equation*}
where $\xi\in\mathbb{C}$ and $f\in W^{1,\infty}(\mathbb{C},\mathbb{C})$.

\section{Conclusion and final comments}
In this work, we have investigated the maximal regularity problem for abstract time-fractional Schrödinger equations governed by a self-adjoint operator with compact resolvent on a Hilbert space, where the fractional derivative is of order $\alpha\in (0,1)$. We first established maximal $L^2$-regularity by exploiting refined estimates for Mittag–Leffler functions with the imaginary argument. In contrast to the classical analysis of subdiffusion equations, our approach does not rely on the complete monotonicity of Mittag–Leffler functions, a property that appears difficult to handle in the presence of the imaginary argument. We then extended these results to maximal $L^p$-regularity for any $p\in (1,\infty)$ by applying the operator-valued Mikhlin multiplier theorem.

We stress that, in contrast to the time-fractional case considered here, the classical Schrödinger equation does not enjoy maximal regularity (see Remark \ref{rmkcs}), which highlights a fundamental difference at this level between the classical and fractional models.

To illustrate the applicability of our maximal regularity results, we have proven the well-posedness for some quasilinear and semilinear Schr\"odinger equations with time-fractional derivatives. Furthermore, our results may serve as a basis for establishing well-posedness of non-autonomous time-fractional Schrödinger equations in the spirit of the recent paper \cite{Ach24}.

As a perspective for future work, we aim to extend the analysis of time-fractional Schrödinger equations to Banach space settings, with particular emphasis on well-posedness, maximal regularity, controllability and inverse problems, in a broader and more general framework; see, for instance, the recent works \cite{BCEM26, CEMYIP25}.

\subsection*{Acknowledgment}
The work was supported by Grant-in-Aid for Challenging Research (Pioneering) 21K18142 of Japan Society for the Promotion of Science.

\end{document}